\documentclass[oneside,english]{amsart}
\usepackage[T1]{fontenc}
\usepackage[latin9]{inputenc}
\usepackage{color}
\usepackage{babel}
\usepackage{verbatim}
\usepackage{enumitem}
\usepackage{amsbsy}
\usepackage{amstext}
\usepackage{amsthm}
\usepackage{amssymb}
\usepackage[unicode=true,pdfusetitle,
 bookmarks=true,bookmarksnumbered=false,bookmarksopen=false,
 breaklinks=false,pdfborder={0 0 0},pdfborderstyle={},backref=false,colorlinks=true]
 {hyperref}

\makeatletter

\newcommand{\noun}[1]{\textsc{#1}}

\numberwithin{equation}{section}
\numberwithin{figure}{section}
  \theoremstyle{plain}
  \newtheorem*{thm*}{\protect\theoremname}
  \theoremstyle{plain}
  \newtheorem*{cor*}{\protect\corollaryname}
\theoremstyle{plain}
\newtheorem{thm}{\protect\theoremname}[section]
  \theoremstyle{definition}
  \newtheorem{defn}[thm]{\protect\definitionname}
  \theoremstyle{remark}
  \newtheorem*{rem*}{\protect\remarkname}
  \theoremstyle{plain}
  \newtheorem{lem}[thm]{\protect\lemmaname}
  \theoremstyle{plain}
  \newtheorem{cor}[thm]{\protect\corollaryname}
  \theoremstyle{plain}
  \newtheorem{prop}[thm]{\protect\propositionname}
  \theoremstyle{definition}
  \newtheorem*{example*}{\protect\examplename}

\usepackage{mathrsfs}

\makeatother

  \providecommand{\corollaryname}{Corollary}
  \providecommand{\definitionname}{Definition}
  \providecommand{\examplename}{Example}
  \providecommand{\lemmaname}{Lemma}
  \providecommand{\propositionname}{Proposition}
  \providecommand{\remarkname}{Remark}
  \providecommand{\theoremname}{Theorem}
\providecommand{\theoremname}{Theorem}

\begin{document}

\title[Transport maps, non-branching and measure rigidity ]{Transport maps, non-branching sets of geodesics and measure rigidity
}

\author{Martin Kell}

\thanks{\\Fachbereich Mathematik, Universität Tübingen, Germany \\ \texttt{e-mail: martin.kell@math.uni-tuebingen.de}}
\begin{abstract}
In this paper we investigate the relationship between a general existence
of transport maps of optimal couplings with absolutely continuous
first marginal and the property of the background measure called essentially
non-branching introduced by Rajala\textendash Sturm (\emph{Calc.Var.PDE}
2014). In particular, it is shown that the qualitative non-degenericity
condition introduced by Cavalletti\textendash Huesmann (\emph{Ann.
Inst. H. Poincar}é\emph{ Anal. Non Linéaire} 2015) implies that any
essentially non-branching metric measure space has a unique transport
maps whenever initial measure is absolutely continuous. This generalizes
a recently obtained result by Cavalletti\textendash Mondino (\emph{Commun.
Contemp. Math.} 2017) on essentially non-branching spaces with the
measure contraction condition $\mathsf{MCP}(K,N)$. 

In the end we prove a measure rigidity result showing that any two
essentially non-branching, qualitatively non-degenerate measures on
a fixed metric spaces must be mutually absolutely continuous. This
result was obtained under stronger conditions by Cavalletti\textendash Mondino
(\emph{Adv.Math.} 2016). It applies, in particular, to metric measure
spaces with generalized finite dimensional Ricci curvature bounded
from below.
\end{abstract}

\maketitle
\hypersetup{linkcolor=black}\tableofcontents{}\hypersetup{linkcolor=red}

\global\long\def\m{\mathsf{m}}
\global\long\def\GTB{\mathsf{(GTB)}}
\global\long\def\sfd{\mathsf{d}}
\global\long\def\restr{\mathsf{restr}}
\global\long\def\Geo{\mathsf{Geo}}
\global\long\def\OptGeo{\mathsf{OptGeo}}
\global\long\def\Opt{\mathsf{Opt}}
\global\long\def\supp{\operatorname{supp}}
\global\long\def\CD{\mathsf{CD}}
\global\long\def\MCP{\mathsf{MCP}}
\global\long\def\RCD{\mathsf{RCD}}
\global\long\def\ND{\mathsf{(ND)}}
\global\long\def\sND{\mathsf{(sND)}}
\global\long\def\IP{\mathsf{(IP)}}
\global\long\def\sIP{\mathsf{(sIP)}}
\global\long\def\cl{\operatorname{cl}}
\global\long\def\interior{\operatorname{int}}

\section{Introduction}

In the theory of optimal transport one of the first questions one
asks is whether an optimal coupling $\pi_{\operatorname{opt}}$ that
minimizes the functional 
\[
\pi\mapsto\int c(x,y)d\pi(x,y)
\]
among all $\pi$ with fixed marginals $(p_{1})_{*}\pi=\mu$ and $(p_{2})_{*}\pi=\nu$
can be written as a coupling induced by a transport map, i.e. whether
there is a measurable map $T:M\to M$ such that $\pi=(\operatorname{id}\times T)_{*}\mu$.
On sufficiently nice spaces this result can be deduced from Rademacher's
Theorem, the (weak) differential structure and the exponential map
whenever $\mu$ is absolutely continuous, see Brenier \cite{Brenier1991}
in the Euclidean setting, McCann \cite{McCann2001} on Riemannian
manifolds, Ambrosio\textendash Rigot \cite{AR2004Heisenberg} and
Figalli\textendash Rigot \cite{FR2010subRiem} on sub-class of sub-Riemannian
manifolds and also Bertrant \cite{Bertrand2008} on Alexandrov spaces.

A first proof for non-smooth non-branching spaces with generalized
Ricci curvature bounded from below was obtained by Gigli \cite{Gigli2012}.
He showed that the non-existence of a transport map would imply that
two disjoint parts of a Wasserstein geodesic whose initial densities
with respect to the background measure $\m$ are bounded would overlap
at intermediate times. This, however, cannot happen for non-branching
spaces (compare Lemma \ref{lem:cyc-mon-intersect-implies-same-length}
below).

Consequently, Gigli's idea was adapted to essentially non-branching
spaces with generalized Ricci curvature bounded from below (see \cite{RS2014NonbranchStrongCD,GRS2016,CM2016TransMapsMCP}).
Here essentially non-branching is a weak version of the non-overlapping
property described above, i.e. it prohibits that initially disjoint
parts of a Wasserstein geodesic overlap at intermediate times whenever
the initial and final measures are absolutely continuous. 

Both Gigli\textendash Rajala\textendash Sturm \cite{GRS2016} and
Cavalletti\textendash Mondino \cite{CM2016TransMapsMCP} had to prove
that there are absolutely continuous interpolations along which a
corresponding interpolation inequality holds. 

Rather than using density bounds of intermediate measures, Cavalletti\textendash Huesmann
\cite{CH2015NBTrans} showed that if not too much mass is lost in
a uniform way when transported towards a fixed point (compare Definition
\ref{def:qual-non-deg}), then, together with the non-overlapping
property implied by non-branching property, the non-existence of transport
maps would yield a contradiction.

Whereas the interpolation inequality implied density bounds, Cavalletti\textendash Huesmann's
approach only relied on an ``easier to measure'' quantity. In this
paper we want to combine this approach with the one of Cavalletti\textendash Mondino
\cite{CM2016TransMapsMCP}. The difficulty is that an arbitrary interpolation
might only ``see'' part of the interpolation points. We avoid this
by proving the following:
\begin{itemize}[label=--]
\item at finitely many fixed times $\{t_{i}\}_{i=1}^{n}$ there are absolutely
continuous interpolations $\mu_{t_{n}}\ll\m$ (see first part of Theorem
\ref{thm:abs-cts-interpolation} and Corollary \ref{cor:abs-cts-interpolation})
\item at a fixed time $t$ the interpolation $\mu_{t}$ ``sees'' $\m$-almost
every possible interpolation point $\Gamma_{t}$ (see second part
of Theorem \ref{thm:abs-cts-interpolation})
\end{itemize}
Whereas the first part gives us the non-overlapping if the space is
essentially non-branching, the second part makes sure that the set-wise
interpolation $\Gamma_{t}$ does not contain a set of positive $\m$-measure
that is not seen by the interpolation $\mu_{t}$.

The remaining parts follow along the line of Cavalletti\textendash Huesmann
\cite{CH2015NBTrans}, more precisely, if the background measure is
\emph{qualitatively non-degenerate}, i.e. there is a function $f:(0,1)\to(0,\infty)$
with $\limsup_{t\to0}f(t)>\frac{1}{2}$ such that for all Borel set
$A$ and all $x\in X$ it holds $\m(A_{t,x})\ge f(t)\m(A)$ where
$A_{t,x}=\{\gamma_{t}\,|\,\gamma_{0}\in A,\gamma_{1}=x\}$, then
\begin{itemize}[label=--]
\item every optimal coupling between $\mu_{0}\ll\m$ and $\mu_{1}=\sum\lambda_{i}\delta_{x_{i}}$
is induced by a transport map (Lemma \ref{lem:no-overlap-non-deg}
and Corollary \ref{cor:maps-between-abs-to-discrete})
\item any $c_{p}$-cyclically monotone set $\Gamma$ satisfies a qualitative
non-degenericity condition, i.e. $\m(\Gamma_{t})\ge f(t)\m(\Gamma_{0})$
(Lemma \ref{lem:p-ess-nb-implies-p-str-non-deg})
\item every optimal coupling between $\mu_{0}\ll\m$ and $\mu_{1}$ is induced
by a transport map (Theorem \ref{thm:MCPimpliesGTB}).
\end{itemize}
As it turns out the general existence of transport maps implies not
only uniqueness of the optimal coupling, but also unique geodesics
between almost all points coupled via the optimal transport, see Lemma
\ref{lem:GTB-superdiff}. This in return gives uniqueness of the interpolation
measures. Furthermore, only assuming a priori the existence of transport
maps from an absolutely continuous initial measure implies that the
space must be essentially non-branching (Proposition \ref{prop:GTB-ENB}).
We call spaces with such an \emph{a priori existence of transport
maps} spaces having \emph{good transport behavior} $\GTB_{p}$, see
Definition \ref{def:GTB}.

We may summarize one of the results of this note as follows.
\begin{thm*}
[see Proposition \ref{prop:GTB-ENB} and Theorem \ref{thm:MCPimpliesGTB}]Let
$(M,d,\m)$ be a metric measure space and assume $\m$ is (uniformly)
qualitatively non-degenerate then the following properties are equivalent:
\begin{enumerate}[label=(\roman*)]
\item $(M,d,\m)$ is $p$-essentially non-branching, i.e. any $p$-optimal
dynamical coupling $\pi\in\mathcal{P}(\Geo_{[0,1]}(M,d))$ with $(e_{0})_{*}\sigma,(e_{1})_{*}\sigma\ll\m$
is concentrated on a set of non-branching geodesics. 
\item for every $\mu_{0},\mu_{1}\in\mathcal{P}_{p}(M)$ with $\mu_{0}\ll\m$
there is a unique $p$-optimal dynamical coupling $\sigma\in\mathcal{P}_{p}(\Geo_{[0,1]}(M,d))$
between $\mu_{0}$ and $\mu_{1}$ and for this coupling $\sigma$
the $p$-optimal coupling $(e_{0},e_{1})_{*}\sigma$ is induced by
a transport map and every interpolation $\mu_{t}=(e_{t})_{*}\sigma$,
$t\in[0,1)$, is absolutely continuous, i.e. $\mu_{t}\ll\m$.
\end{enumerate}
\end{thm*}
It is well-known that on smooth spaces, there is an abundance of measures
satisfying all but the last property of the second statement in the
theorem above, in particular, they have the good transport behavior
$\GTB_{p}$. 

If interpolations We call the last property of the second statement
in the theorem above \emph{strong interpolation property} $\sIP_{p}$.
Note that two reference measures $\m_{1}$ and $\m_{2}$ have the
strong interpolation property and $\mu_{0}\ll\m_{1}$ and $\mu_{1}\ll\m_{2}$
then the interpolation $(e_{t})_{*}\sigma$ must be absolutely continuous
with respect to both $\m_{1}$ and $\m_{2}$ implying $\m_{1}$ and
$\m_{2}$ cannot be mutually singular. This property can be used to
obtain the following measure rigidity theorem. 
\begin{thm*}
[Measure Rigidity] Any two measures on a complete separable metric
space $(M,d)$ which are $p$-essentially non-branching and qualitatively
non-degenerate must be mutually absolutely continuous.
\end{thm*}
Under a more involved inversion property and a stronger qualitative
non-degenericity with $\limsup_{t\to0}f(t)=1$ such a statement was
obtained by Cavalletti\textendash Mondino \cite{CM2016MeasRigid}. 

By \cite{RS2014NonbranchStrongCD,AGS2014RCD} the statement applies
in particular to spaces with finite dimensional Ricci curvature bounded
from below.
\begin{cor*}
If both $(M,d,\m_{1})$ and $(M,d,\m_{2})$ are $\RCD^{*}(K,N)$-spaces
with $N\in[1,\infty)$ then $\m_{1}$ and $\m_{2}$ must be mutually
absolutely continuous. 
\end{cor*}
A slightly different kind of measure rigidity of the background measure
$\m$ for metric measure spaces satisfying the $\RCD(K,N)$-condition
with $N\in[1,\infty)$ was obtain independently by Gigli\textendash Pasqualetto
\cite{GP2016RCDMeasRigid} and by Mondino and the author \cite{KM2017volRCD}
using the weak converse of Rademacher's theorem proven in \cite{DePR2016LipDiffRigid}.

The section dealing with the measure rigidity theorem (Section \ref{sec:Measure-rigidity})
only relies on few properties which are restated in the beginning
of that section and can be read independently of the rest of this
note. We also present a proof of the rigidity theorem that relies
on a bounded density property and applies to strong $\CD_{p}(K,\infty)$-spaces
with the strong interpolation property, see Theorem \ref{thm:measure-rigidity-ess-nb-CD}.

\section{Preliminaries}

Throughout this paper we always assume that $(M,d,\m)$ is a\emph{
geodesic metric measure space}, i.e. $(M,d)$ is a complete separable
geodesic metric space and $\m$ is measure on $M$ which is finite
on bounded sets. It is \emph{proper} if every closed bounded set $B\subset M$
is compact.

\subsection*{Selection Dichotomy and Disintegration Theorem}

We present two technical results which help to classify couplings
that are induced by a transport map. The first can be obtained by
combining the Measurable Selection Theorem and Lusin's Theorem. This
form of the selection dichotomy was used by Cavalletti\textendash Huesmann
\cite{CH2015NBTrans} to select $p$-optimal couplings that overlap
at the initial measure. 

In the following for a set $\Gamma\subset M\times M$ we let $\Gamma(x)=\{y\in M\,|\,(x,y)\in\Gamma\}$.
We say a map $T:M\to M$ is a selection of $\Gamma$ is $(x,T(x))\in\Gamma$
for all $x\in p_{1}(\Gamma)$.
\begin{thm}
[Selection Dichotomy of Sets]\label{thm:selection-dichotomy-set}Assume
$\mu$ is a probability measure on $M$ and $\Gamma\subset M\times M$
a Borel set with $\mu(p_{1}(\Gamma))=1$. Then exactly one of the
following holds:
\begin{enumerate}[label=(\roman*)]
\item For $\mu$-almost all $x\in M$ the set $\Gamma(x)$ contains exactly
one element. Furthermore, if $\pi\in\mathcal{P}(M\times M)$ with
$\supp\pi\subset\Gamma$ and $(p_{1})_{*}\pi=\mu$ then $\pi=(\operatorname{id}\times T)_{*}\mu$
for a $\mu$-measurable selection $T:M\to M$ of $\Gamma$. In particular,
$T$ is unique up to $\mu$-measure zero and $\pi$ is unique among
all measure $\tilde{\pi}\in\mathcal{P}(M\times M)$ with support in
$\Gamma$ and first marginal $\mu$.
\item There are a compact set $K\subset\supp\mu$ with $\mu(K)>0$ and two
$\mu$-measurable selections $T_{1},T_{2}:M\to M$ of $\Gamma$ which
are continuous when restricted to $K$ and $T_{1}(K)\cap T_{2}(K)=\varnothing$.
Furthermore, if the function 
\[
\varphi_{\Gamma}(x)=\begin{cases}
\sup_{(x,y)\in\Gamma}d(x,y)-\inf_{(x,y)\in\Gamma}d(x,y) & x\in p_{1}(\Gamma)\\
0 & \text{otherwise}.
\end{cases}
\]
is positive on a set of positive $\mu$-measure then $K$ can be chosen
such that for some $\delta>0$ 
\[
\sup_{(x,y_{1})\in K\times T_{1}(K)}d(x,y_{1})+\delta\le\inf_{(x,y_{2})\in K\times T_{2}(K)}d(x,y_{2}).
\]
\end{enumerate}
\end{thm}
\begin{proof}
It is easy to see that the conditions are mutually exclusive. Indeed,
the measures $(\operatorname{id}\times T_{1})_{*}\mu$ and $(\operatorname{id}\times T_{2})_{*}\mu$
are distinct and supported on $\Gamma$.

By the Measurable Selection Theorem there is a $\mu$-measurable map
$T$ such that $(x,T(x))\in\Gamma$ for all $x\in p_{1}(\Gamma)$.
Using Lusin's Theorem one can show that there is a Borel set $\Omega\subset M$
of full $\mu$-measure such that 
\[
\operatorname{graph}_{\Omega}T:=\{(x,T(x)\,|\,x\in\Omega\}
\]
is a Borel subset of $M\times M$. Thus $\Gamma^{'}=\Gamma\backslash\operatorname{graph}_{\Omega}T$
is a Borel set and there is a $\mu$-measurable selection $S:M\to M$
with $(x,S(x))\in\Gamma^{'}$ for all $x\in p_{1}(\Gamma^{'})$. Since
$p_{1}(\Gamma^{'})$ is $\mu$-measurable, we can redefine $S$ outside
of $p_{1}(\Gamma^{'})$ and assume $T(x)=S(x)$ for all $x\in M\backslash p_{1}(\Gamma^{'})$. 

If $\mu(p_{1}(\Gamma^{'}))=0$ then $\mu(p_{1}(\Gamma^{'})\cap\Omega)=0$
and thus $S(x)=T(x)$ for $\mu$-almost all $x\in M$. In particular,
the first case holds. 

Otherwise, by Lusin's Theorem there is a compact set $K_{1}\subset\supp\mu\cap p_{1}(\Gamma')\cap\Omega$
with $\mu(K_{1})>0$ such that $T$ and $S$ are continuous on $K_{1}$
and $S(x)\ne T(x)$ for $x\in K_{1}$. Thus for sufficiently small
$r>0$ and a fixed $x_{0}\in K_{1}$, it holds 
\[
T(x)\ne S(x')\quad\text{for all }x,x'\in\bar{B}_{r}(x_{0})\cap K_{1}.
\]
In case $\mu(\{\varphi_{\Gamma}=0\})=0$ we can choose $K=\bar{B}_{r}(x_{0})\cap K_{1}$,
$T_{1}=T$ and $T_{2}=S$ and conclude. 

If $\mu(\{\varphi_{\supp\pi}>0\})>0$ then there are $\epsilon>0$
and a compact set $K_{2}\subset\supp\mu$ with $\mu(K_{2})>0$ and
$\varphi_{\supp\pi}(x)>\epsilon$ for all $x\in K_{2}$. Note also
that the sets 
\begin{align*}
\Gamma^{+} & =\{(x,y)\in\Gamma\,|\,x\in K_{2},d(x,y)\ge\sup_{(x,y')\in\Gamma}d(x,y')-\frac{\epsilon}{2}\}\\
\Gamma^{-} & =\{(x,y)\in\Gamma\,|\,x\in K_{2},d(x,y)\le\inf_{(x,y')\in\Gamma}d(x,y')+\frac{\epsilon}{2}\}
\end{align*}
are non-empty Borel subsets of $\Gamma$ with $p_{1}(\Gamma^{+})=p_{1}(\Gamma^{-})=K_{2}$.
Thus there are two $\mu$-measurable selection $T^{+}$ and $T^{-}$
with $(x,T^{\pm}(x))\in\Gamma^{\pm}$ for all $x\in K_{2}$. As above
we may assume that $T^{\pm}$ agree with $T$ outside of $K_{2}$. 

Choose another compact $K_{3}\subset K_{2}$ such that the maps $T^{\pm}$
are continuous on $K_{3}$. In particular, for some $x_{0}\in K_{3}$
and sufficiently small $r>0$ it holds 
\[
d(x,T^{-}(x'))+\delta\le d(x,T^{+}(x''))\qquad\text{for all }x,x',x''\in\bar{B}_{r}(x_{0})\cap K_{3}
\]
To conclude observe that the compact set $K=\bar{B}_{r}(x_{0})\cap K_{3}$
and the maps $T_{1}=T^{-}$ and $T_{2}=T^{+}$ satisfy the last part
of the second statement.
\end{proof}
Note that in general the second possibility of the Selection Dichotomy
above does not say anything about the relationship of the measures
$(\operatorname{id}\times T_{i})_{*}\mu$ and a fixed measure $\pi\in\mathcal{P}(M\times M)$
with $(p_{1})_{*}\pi=\mu$ and $\supp\pi\subset\Gamma$. More precisely,
in general, $(T_{i})_{*}\mu$ might be singular with respect to $(p_{2})_{*}\pi$,
or more generally, it is possible that $(T_{i})_{*}\mu\bot\m$ even
if $(p_{2})_{*}\pi\ll\m$. 

The following lemma is a more general version of the Selection Dichotomy
and can be extracted from Gigli's work \cite[Proof of Theorem 3.3]{Gigli2012}.
It shows that any measure $\pi$, regarded as a generalized transport
map $\int\delta_{x}\otimes\mu_{x}d\mu(x)$, is either already induced
by a transport map, i.e. $\mu_{x}=\delta_{T(x)}$, or can be decomposed
into (at least) two partial transport with target transport on a compact
set $K$ of positive $\mu$-measure. We give a simpler proof relying
on the Selection Dichotomy for Sets.

First, recall the the statement of the Disintegration Theorem. Let
$(X,d)$ and $(Y,d)$ be two complete separable metric spaces. Denote
the Borel $\sigma$-algebra of $X$ and $Y$ by $\mathcal{B}(X)$
and resp. $\mathcal{B}(Y)$.
\begin{defn}
[Disintegration over $S$] Let $\sigma$ a probability measure on
$X$, $S:X\to Y$ a Borel map and $\varpi=S_{*}\sigma$. An assignment
$\sigma:\mathcal{B}(X)\times Y\to[0,1]$, denoted $(B,y)\mapsto\sigma_{y}(B)$,
is called \emph{a disintegration of $\sigma$ over $S$ }if 
\begin{enumerate}
\item $\sigma_{y}(\cdot)$ is a probability measure on $X$ for all $y\in Y$
\item $y\mapsto\mu_{y}(B)$ is $\varpi$-measurable for all Borel sets $B\in\mathcal{B}(X)$. 
\item $\mu_{y}(S^{-1}(y))=1$ for all $y\in Y$ 
\end{enumerate}
and it holds 
\[
\sigma(C\cap S^{-1}(B))=\int_{C}\sigma_{y}(B)d\varpi(y).
\]
Regarding the assignment $y\mapsto\sigma_{y}=\sigma_{y}(\cdot)$ as
a map from $Y$ to $\mathcal{P}(X)$ we abbreviate this as 
\[
\sigma=\int\sigma_{y}d\varpi(y).
\]
\end{defn}
\begin{rem*}
(1) A disintegration of $\sigma$ over $S$ as above is usually called
a disintegration of $\sigma$ which strongly consistent with $S$. 

(2) If $\pi\in\mathcal{P}(M\times M)$ for some complete separable
metric space $M$ with $\mu=(p_{1})_{*}\pi$, then any disintegration
$\pi=\int\pi_{x}d\mu(x)$ must satisfy $\supp\pi_{x}\subset\{x\}\times M$
and hence $\pi_{x}=\delta_{x}\otimes\mu_{x}$ for a measure $\mu_{x}\in\mathcal{P}(M)$.
\end{rem*}
The following theorem can be deduced from the general Disintegration
Theorem \cite[Section 452]{Fremlin2006}.
\begin{lem}
[Disintegration Theorem] For every $\sigma\in\mathcal{P}(X)$ and
every Borel map $S:X\to Y$ there exists a unique disintegration $\sigma_{\cdot}(\cdot)$
of $\sigma$ over $S$ and for any other disintegration $\tilde{\sigma}_{\cdot}(\cdot)$
of $\sigma$ over $S$ it holds $\sigma_{y}(\cdot)=\tilde{\sigma}_{y}(\cdot)$
for $S_{*}\sigma$-almost all $y\in Y$.
\end{lem}
The theorem allows us to say that up to a $\mu$-null set $\sigma=\int\sigma_{x}d\mu$
is \emph{the} disintegration of $\sigma$ over $S$. 
\begin{thm}
[Selection Dichotomy for Measures]\label{thm:selection-dichotomy-product-measures}Let
$\pi$ be a probability measure on $M\times M$ and $\mu=(p_{1})_{*}\pi$.
Then exactly one of the following holds:
\begin{enumerate}[label=(\roman*)]
\item There is a $\mu$-measurable map $T:M\to M$ such that $\pi(\operatorname{graph}T)=1$.
\item There are compact set $K\subset M$ and two closed bounded sets $A_{1},A_{2}\subset M$
with $A_{1}\cap A_{2}=\varnothing$ and
\[
\pi(K\times A_{1}),\pi(K\times A_{2})>0.
\]
Furthermore, there are two measure $\pi_{1},\pi_{2}\in\mathcal{P}(M\times M)$
with $\frac{1}{\mu(K)}\mu\big|_{K}=(p_{1})_{*}\pi_{1}=(p_{1})_{*}\pi_{2}$
and $\pi_{1},\pi_{2}\ll\pi$. 
\end{enumerate}
\end{thm}
\begin{rem*}
(1) The construction shows that for some $\epsilon>0$ it holds 
\[
\left(A_{1}\right)_{\epsilon}\cap\left(A_{2}\right)_{\epsilon}=\varnothing
\]
where $A_{\epsilon}=\bigcup_{x\in A}B_{\epsilon}(x)$ for a set $A\subset M$.
Furthermore, as above it is possible to choose $K$, $A_{1}$ and
$A_{2}$ such that for some $\delta>0$ 
\[
\sup_{(x,y_{1})\in K\times A_{1}}d(x,y)+\delta\le\inf_{(x,y_{1})\in K\times A_{1}}d(x,y).
\]

(2) If $\pi=\int\delta_{x}\otimes\mu_{x}d\mu(x)$ is the disintegration
over $p_{1}$ and $\pi$ is not induced by a map then for $\mu$-almost
all $x\in K$ the measure $\mu_{x}$ is not a delta measure. Indeed,
for $\mu$-almost all $x\in K$ it holds 
\[
\mu_{x}=\mu_{x}\big|_{A_{1}}+\mu_{x}\big|_{A_{2}}+\mu_{x}\big|_{M\backslash(A_{1}\cup A_{2})}
\]
and the choice of $K$ shows that $\mu_{x}\big|_{A_{1}}$ and $\mu_{x}\big|_{A_{2}}$
are non-trivial for $\mu$-almost all $x\in K$.
\end{rem*}
\begin{proof}
We apply Theorem \ref{thm:selection-dichotomy-set} to $\Gamma=\supp\pi$.
If the first option of the dichotomy holds then $\pi=(\operatorname{id}\times T)_{*}\mu$
and thus $\pi(\operatorname{graph}T)=1$.

Otherwise let $K$, $T_{1}$ and $T_{2}$ as in the second possibility
of Theorem \ref{thm:selection-dichotomy-set}. We may restrict $K$
further and assume $\supp(\mu\big|_{K})=K$. 

We claim that for all $\epsilon>0$ and $i=1,2$ it holds
\[
\pi(K\times(T_{i}(K))_{\epsilon})>0.
\]
Indeed, note that $\pi\big|_{K}=\pi(\cdot\cap(K\times M))\ne\mathbf{0}$,
so that $(x,T_{1}(x)),(x,T_{2}(x))\in\supp(\pi\big|_{K})=\supp\pi\cap(K\times M)$.
Because $B_{\epsilon}^{M\times M}(x,y)\subset B_{\epsilon}(x)\times B_{\epsilon}(y)$
for $i=1,2$ it holds 
\begin{align*}
0 & <\pi\big|_{K}(B_{\epsilon}^{M\times M}(x,T_{i}(x)))\\
 & \le\pi((B_{\epsilon}(x)\cap K)\times B_{\epsilon}(T(x)))\\
 & \le\pi(K\times B_{\epsilon}(T(x)))\le\pi(K\times(T_{i}(K))_{\epsilon}).
\end{align*}
Since $T_{1}$ and $T_{2}$ are continuous on $K$ and $K$ is compact
there is an $\epsilon>0$ such that 
\[
(T_{1}(K))_{2\epsilon}\cap(T_{2}(K))_{2\epsilon}=\varnothing.
\]
Choosing $A_{1}=\cl(T_{1}(K))_{\epsilon}$ and $A_{1}=\cl(T_{2}(K))_{\epsilon}$
gives first part of the claim.

To obtain the second part, note that there are a $\delta>0$ and compact
$K'$ of positive $\mu$-measure such that 
\[
\mu_{x}(A_{1}),\mu_{x}(A_{2})\in(\delta,1-\delta)\qquad\text{for all }x\in K'.
\]
Restricting $K'$ again, assume $K'=\supp(\mu\big|_{K'})$ and define
two non-trivial measures $\pi_{1},\pi_{2}\in\mathcal{P}(M\times M)$
as follows 
\begin{align*}
\pi_{1} & =\frac{1}{\mu(K')}\int_{K'}\frac{1}{\mu_{x}(A_{1})}\delta_{x}\otimes\mu_{x}d\mu(x)\\
\pi_{2} & =\frac{1}{\mu(K')}\int_{K'}\frac{1}{\mu_{x}(A_{2})}\delta_{x}\otimes\mu_{x}d\mu(x).
\end{align*}
It is easy to see that $\pi_{1},\pi_{2}\ll\pi$ and $\frac{1}{\mu(K')}\mu\big|_{K'}=(p_{1})_{*}\pi_{1}=(p_{1})_{*}\pi_{2}$
which proves the claim.
\end{proof}
For completeness we present the following more general form of the
Selection Dichotomy.
\begin{cor}
[General Selection Dichotomy] Assume $(X,d)$ and $(Y,d)$ are complete
separable metric spaces. Let $\sigma$ be a measure on $X$ and $S:X\to Y$
a Borel map. Then exactly one of the following holds:
\begin{enumerate}[label=(\roman*)]
\item There is a measurable map $T:Y\to X$ such that $S(T(y))=y$ and
$T_{*}\varpi=\sigma$ where $\varpi=S_{*}\sigma$. In particular,
the disintegration of $\sigma$ via $S$ is given by 
\[
\sigma=\int\delta_{T(y)}d\varpi(y).
\]
\item There are a compact set $K\subset X$ and two closed bounded sets
$A_{1},A_{2}\subset X$ with $A_{1}\cap A_{2}=\varnothing$ and
\[
\sigma_{x}\big|_{A_{i}}\ne0,\;x\in K,i=1,2.
\]
In particular, for $\varpi$-almost all $x\in K$ the measures $\sigma_{x}$
are not delta measures.
\end{enumerate}
\end{cor}
\begin{proof}
Just note that the proofs above did not rely on the product structure
of $M\times M$ and that $p_{1}$ is a projection. Thus replace $M\times M$
by $X$ and $p_{1}$ by $S$ we can follow the proofs above line by
line. 
\end{proof}

\subsection*{Wasserstein spaces on geodesic spaces}

Let $(X,d)$ be a complete, separable metric space. A map $\gamma:[0,1]\to X$
satisfying 
\[
d(\gamma_{t},\gamma_{s})=|t-s|d(\gamma_{0},\gamma_{1})\qquad\text{for }t,s\in[0,1]
\]
is called a \emph{geodesic connecting $\gamma_{0}$ and $\gamma_{1}$}.
Note that our terminology implies that any geodesic is the curve of
minimal length between its endpoints. Denote by $\Geo_{[0,1]}(X,d)$
the set of geodesics. On $\Geo_{[0,1]}(X,d)$ there are natural evaluation
maps $e_{t}:\Geo_{[0,1]}(M,d)\to M$, $t\in[0,1]$, defined by $e_{t}:\gamma\mapsto\gamma_{t}$.
Denote the length of a geodesic $\gamma$ by $\ell(\gamma):=d(\gamma_{0},\gamma_{1})$
and define a restriction map $\restr_{s,t}:\Geo_{[0,1]}(M,d)\to\Geo_{[0,1]}(M,d)$
for all $0\le s,t\le1$ by 
\[
(\restr_{s,t}\gamma)(r)=\gamma_{s+(t-s)r}.
\]
We say the metric space $(X,d)$ is a \emph{geodesic metric space
}if between each\emph{ }$x,y\in X$ there is a geodesic connecting
$x$ and $y$, i.e. \emph{$(e_{0},e_{1})(\Geo_{[0,1]}(X,d))=X\times X$.}

In the following we introduce the main concepts used from the theory
of optimal transport. For a comprehensive introduction we refer the
reader to Villani's book \cite{Villani2008OptTrans}. Recall that
$(M,d)$ is a complete separable geodesic metric space. Let $\mathcal{P}(M)$
be the set of probability measures on $M$ and for a fixed $x_{0}\in M$
let 
\[
\mathcal{P}_{p}(M)=\left\{ \mu\in\mathcal{P}(M)\,|\,\int d(x,x_{0})^{p}d\mu(x)\right\} 
\]
the space of probability measures with \emph{finite $p$-th moment}.
On $\mathcal{P}_{p}(M)$ we define the \emph{$p$-Wasserstein metric
$W_{p}$} as follows
\[
W_{p}(\mu_{0},\mu_{1})=\left(\inf_{\pi\in\Pi(\mu_{0},\mu_{1})}\int d(x,y)^{p}d\pi(x,y)\right)^{\frac{1}{p}}
\]
where $\Pi(\mu_{0},\mu_{1})$ is the set of $\pi\in\mathcal{P}(M\times M)$
with $(p_{1})_{*}\pi=\mu_{0}$ and $(p_{2})_{*}\pi=\mu_{1}$. This
defines a complete metric on $\mathcal{P}_{p}(M)$ with a topology
which is strictly stronger than the subspace topology induced by $\mathcal{P}_{p}(M)\subset\mathcal{P}(M)$
unless $(M,d)$ is bounded. We call the convergence induced by the
subspace topology \emph{weak convergence}.

One can show that for each $\mu_{0},\mu_{1}\in\mathcal{P}_{p}(M)$
there is a $\pi_{\operatorname{opt}}\in\Pi(\mu_{0},\mu_{1})$ such
that 
\[
W_{p}(\mu_{0},\mu_{1})=\left(\int d(x,y)^{p}d\pi_{\operatorname{opt}}(x,y)\right)^{\frac{1}{p}}.
\]
In this case we say $\pi_{\operatorname{opt}}$ is a \emph{$p$-optimal
coupling}. Let $\Opt_{p}(\mu_{0},\mu_{1})$ denote the set of all
$p$-optimal couplings between $\mu_{0}$ and $\mu_{1}$. A general
measure $\pi\in\mathcal{P}(M\times M)$ is said to be $p$-optimal
if it is a $p$-optimal coupling between $(e_{0})_{*}\pi$ and $(e_{1})_{*}\pi$.

Since $(M,d)$ is geodesic it is possible to show that $(\mathcal{P}_{p}(M),W_{p})$
is geodesic as well. Just note that $(x,y)\mapsto(e_{0},e_{1})^{-1}(x,y)\subset\Geo_{[0,1]}(M,d)$
is a measurable closed-valued map and any selection $\mathcal{T}$
will lift a coupling $\pi$ to a dynamical coupling $\sigma=\mathcal{T}_{*}\pi\in\mathcal{P}(\Geo_{[0,1]}(M,d))$
between two measure $\mu_{0}$ and $\mu_{1}$. If $\pi$ is $p$-optimal
than we say $\sigma\in\mathcal{P}(\Geo_{[0,1]}(M,d))$ is a $p$-optimal
dynamical coupling. Now one may readily verify that $t\mapsto(e_{t})_{*}\sigma$
is a geodesic connecting $(e_{0})_{*}\sigma$ and $(e_{1})_{*}\sigma$.
Denote the set of of $p$-optimal dynamical couplings between $\mu_{0}$
and $\mu_{1}$ by $\OptGeo_{p}(\mu_{0},\mu_{1})$. 

Also note each geodesic $t\mapsto\mu_{t}$ in $\mathcal{P}_{p}(M)$
is induced by a (unique) measure $\sigma\in\mathcal{P}(\Geo_{[0,1]}(M,d))$
such that $(e_{t})_{*}\sigma=\mu_{t}$. In this case it is easy to
see that $(e_{t},e_{s})_{*}\sigma$ is a $p$-optimal coupling between
$\mu_{t}$ and $\mu_{s}$. 

Recall that disintegrating a dynamical coupling $\sigma$ over $(e_{0},e_{1}):\Geo_{[0,1]}(M,d)\to M\times M$
shows that 
\[
\sigma=\int\sigma_{x,y}d\pi(x,y)
\]
where $\pi=(e_{0},e_{1})_{*}\sigma$ and $(x,y)\mapsto\sigma_{x,y}$
is a measurable assignment of dynamical couplings between $\delta_{x}$
and $\delta_{y}$. Similarly, we can disintegrate $\sigma$ over $e_{0}$
to obtain $\sigma=\int\sigma_{x_{0}}d\mu_{0}(x_{0})$ such that for
$\mu_{0}$-almost all $x_{0}\in M$ the probability measure $\sigma_{x_{0}}$
is a dynamical coupling between $\delta_{x_{0}}$ and a probability
measure $\mu_{x_{0}}$ with $\pi=\int\delta_{x_{0}}\otimes\mu_{x_{0}}d\mu(x_{0})$.
Furthermore, if $\sigma$ is $p$-optimal then $\sigma_{x_{0}}$ is
$p$-optimal for $\mu_{0}$-almost all $x_{0}\in M$. 

The following is the well-known restriction property of optimal couplings,
see \cite{Villani2008OptTrans}. Compare the following notation also
the the concept push-forward via a plan, see e.g. \cite[Definition 2.1]{AGS2014RCD}.
\begin{lem}
\label{lem:restriction-property}Assume $\mu_{0},\mu_{1}\in\mathcal{P}_{p}(M)$
and $\sigma$ is a $p$-optimal dynamical coupling between $\mu_{0}$
and $\mu_{1}$. Then for $f:M\times M\to[0,\infty)$ with $\lambda=\int fd\pi\in(0,\infty)$
the measure 
\[
\sigma_{f}=\lambda^{-1}\int\sigma_{x,y}f(x,y)d\pi(x,y)
\]
is a $p$-optimal coupling between $\mu_{0}^{f}$ and $\mu_{1}^{f}$
where
\begin{align*}
\mu_{0}^{f} & =(p_{1})_{*}\sigma_{f}\\
\mu_{1}^{1} & =(p_{2})_{*}\sigma_{f}.
\end{align*}
Furthermore, if, in addition $f\le1$,$\int fd\pi\in(0,1)$ and $\tilde{\sigma}_{f}$
is another $p$-optimal dynamical coupling between $\mu_{0}^{f}$
and $\mu_{1}^{f}$ then 
\[
\tilde{\sigma}=\lambda\tilde{\sigma}_{f}+(1-\lambda)\sigma_{1-f}
\]
is also a $p$-optimal coupling between $\mu_{0}$ and $\mu_{1}$.
\end{lem}
\begin{rem*}
(1) If $f$ is a function depending only on the first coordinate then
$\mu_{0}^{f}=\lambda^{-1}f\mu_{0}$.

(2) If $\Gamma\subset M\times M$ is Borel set with $\sigma(\Gamma)>0$
then we write $\sigma_{\Gamma}=\sigma_{\chi_{\Gamma}}=\frac{1}{\pi(\Gamma)}\sigma\big|_{\hat{\Gamma}}$
where $\hat{\Gamma}=(e_{0},e_{1})^{-1}(\Gamma)$.
\end{rem*}
\begin{proof}
The first part follows from the restriction property of optimal transport
and the second from linearity of the cost functional 
\[
\sigma\mapsto\int d(\gamma_{0},\gamma_{1})^{p}d\sigma(\gamma)
\]
and the fact that $\tilde{\sigma}$ is still a dynamical coupling
between $\mu_{0}$ and $\mu_{1}$.
\end{proof}
It is easy to see that whenever there are two distinct $p$-optimal
coupling $\pi_{1}$ and $\pi_{2}$ between $\mu_{0}$ and $\mu_{1}$
then there are at least two distinct $p$-optimal dynamical coupling
$\sigma_{1}$ and $\sigma_{2}$ between $\mu_{0}$ and $\mu_{1}$.
By convexity this would actually give a continuum of $p$-optimal
(dynamical) couplings. If, however, the dynamical coupling is unique
we get the following for restrictions of the endpoints.
\begin{cor}
\label{cor:uniqueness-of-restrictions}Let $\mu_{0}$ and $\mu_{1}$
be two measures in $\mathcal{P}_{p}(M)$. If there is a unique $p$-optimal
dynamical coupling $\sigma$ between $\mu_{0}$ and $\mu_{1}$ then
for any $f\in L^{\infty}(\pi)$ the $p$-optimal dynamical coupling
$\sigma_{f}$ is unique between $\mu_{0}^{f}$ and $\mu_{1}^{f}$.
\end{cor}
\begin{proof}
Note that $\sigma_{f}=\sigma_{cf}$ for all $c>0$. Thus if $f\in L^{\infty}(\pi)$
then it is possible to replace $f$ by $\frac{1}{\|f\|_{\infty}}f$
and assume without loss of generality $f\le1$ and $\int fd\m\in(0,1)$.
Thus the dynamical coupling $\tilde{\sigma}:=\lambda\tilde{\sigma}_{f}+(1-\lambda)\sigma_{1-f}$
is $p$-optimal between $\mu_{0}$ and $\mu_{1}$ whenever $\tilde{\sigma}_{f}$
is $p$-optimal between $\mu_{0}^{f}$ and $\mu_{1}^{f}$. Furthermore,
if $\tilde{\sigma}_{f}$ is distinct from $\sigma_{f}$ then $\tilde{\sigma}$
is also distinct from $\sigma$ proving the claim.
\end{proof}

\subsection*{Non-branching geodesics}

Given a set $\Gamma\subset M\times M$ we frequently use the following
abbreviations 
\[
\hat{\Gamma}=(e_{0},e_{1})^{-1}(\Gamma)
\]
and for $t,s\in[0,1]$
\begin{align*}
\Gamma_{t} & =e_{t}(\hat{\Gamma})\\
\Gamma_{t,s} & =(e_{t},e_{s})(\hat{\Gamma}).
\end{align*}
Thus $\hat{\Gamma}$ is the set of geodesics with endpoints $(x,y)\in\Gamma$
and $\Gamma_{t}$ is the set of $t$-midpoints, where $z$ is a $t$-midpoint
of $x$ and $y$ if $\gamma_{t}=z$ for a geodesic connecting $x$
and $y$.

For a set $A\subset M$ and $x\in M$ we also use the abbreviation
\[
A_{t,x}=\{\gamma_{t}\,|\,\text{for some }\gamma\in\Geo_{[0,1]}(M,d)\,\text{ with }\gamma_{0}\in A\,\text{ and }\gamma_{1}=1\}.
\]

Let $\mathsf{L}$ be a subset of geodesics, then we denote by $\mathsf{L}^{-1}$
the set of reversed geodesics, i.e. 
\[
\mathsf{L}^{-1}=\{t\mapsto\gamma_{1-t}\,|\,\gamma\in\mathsf{L}\}.
\]
 Similarly, let $\Gamma^{-1}=\{(y,x)\,|\,(x,y)\in\Gamma\}$. It is
easy to see that $(\Gamma^{-1})^{\wedge}=(\hat{\Gamma})^{-1}$. 
\begin{defn}
[non-branching set]A set of geodesics $\mathsf{L}\subset\operatorname{Geo}_{[0,1]}(M,d)$
is \emph{non-branching to the right} if for all $\gamma,\eta\in\mathsf{L}$
with $\restr_{0,t}\gamma=\restr_{0,t}\eta$ it holds $\gamma\equiv\eta$.
Similarly, $\mathsf{L}$ is \emph{non-branching to the left} if $\mathsf{L}^{-1}$
is non-branching to the right. Furthermore, $\mathsf{L}$ is non-branching
if it is both non-branching to the left and to the right.
\end{defn}
\begin{rem*}
Non-branching to the left is the same as Rajala\textendash Sturm's
non-branching condition \cite[Section 2.2]{RS2014NonbranchStrongCD}.
This lack of symmetry is irrelevant in their study as essentially
non-branching is a symmetric condition when one changes initial and
final points (see below).
\end{rem*}
The following is well-known and follows from strict convexity of $r\mapsto r^{p}$
and the triangle inequality \cite{CH2015NBTrans,Kell2015}.%

\begin{lem}
\label{lem:cyc-mon-intersect-implies-same-length}Assume for $\gamma,\eta\in\Geo_{[0,1]}(M,d)$
it holds 
\[
d^{p}(\gamma_{0},\gamma_{1})+d^{p}(\eta_{0},\eta_{1})\le d^{p}(\gamma_{0},\eta_{1})+d^{p}(\eta_{0},\gamma_{1}).
\]
Then $\gamma_{t}=\eta_{t}$ for some $t\in(0,1)$ implies $\ell(\gamma)=\ell(\eta)$
and if $(M,d)$ is, in addition, non-branching then $\gamma\equiv\eta$.
\end{lem}
Finally recall some properties of Wasserstein geodesics on essential
non-branching spaces. We collect the result which can be deduced from
\cite{RS2014NonbranchStrongCD,CM2016MeasRigid}. The reader may consult
the appendix for a proof of the theorem and its corollary.
\begin{defn}
[$p$-essentially non-branching] A metric measure space $(M,d,\m)$
is \emph{$p$-essentially non-branching} if for all $\mu_{0},\mu_{1}\in\mathcal{P}_{p}(M)$
with $\mu_{0},\mu_{1}\ll\m$, any optimal dynamical coupling $\sigma\in\OptGeo_{p}(\mu_{0},\mu_{1})$
is concentrated on a set of non-branching geodesics, i.e. there is
a measurable set $\mathsf{L}\subset\Geo_{[0,1]}(M,d)$ such $\sigma(\mathsf{L})=1$.
For brevity we say a measure $\m$ is $p$-essentially non-branching
if $(M,d,\m)$ is a metric measure space which is $p$-essentially
non-branching.
\end{defn}
\begin{rem*}
The property essentially non-branching introduced in \cite{RS2014NonbranchStrongCD}
is equivalent the property $2$-essentially non-branching.
\end{rem*}
\begin{thm}
\label{thm:ess-nb-summary}Assume $(M,d,\m)$ is a $p$-essentially
geodesic measure space. Then for every $p$-optimal dynamical coupling
$\sigma\in\mathcal{P}_{p}(M)$ with $(e_{0})_{*}\sigma,(e_{1})_{*}\sigma\ll\m$
the following holds:
\begin{enumerate}[label=(\roman*)]
\item For each $t\in(0,1)$ there is a Borel map $\mathsf{T}_{t}:M\to\Geo_{[0,1]}(M,d)$
such that the disintegration of $\sigma$ over $e_{t}$ 
\[
\sigma=\int\delta_{\mathsf{T}_{t}(x)}d\mu_{t}(x)
\]
where $\mu_{t}=(e_{t})_{*}\sigma$.
\item For each $t\in(0,1)$ there is a measurable set of geodesics $\mathsf{L}\subset\Geo_{[0,1]}(M,d)$
with $\sigma(\mathsf{L})=1$ and whenever $\gamma_{t}=\eta_{t}$ for
$\gamma,\eta\in\mathsf{L}$ then $\gamma\equiv\eta$.
\item For all disjoint Borel sets $\Gamma^{(1)},\Gamma^{(2)}\subset M\times M$
of positive $(e_{0},e_{1})_{*}\sigma$-measure the $t$-midpoints
of the restricted geodesics $s\mapsto(e_{s})_{*}\sigma_{\Gamma^{(i)}}$,
$i=1,2$, are mutually singular, i.e. 
\[
(e_{t})_{*}\sigma_{\Gamma^{(1)}}\bot(e_{t})_{*}\sigma_{\Gamma^{(2)}}.
\]
\end{enumerate}
\end{thm}
\begin{cor}
\label{cor:ess-nb-intermediate-transport}Assume $(M,d,\m)$ is $p$-essentially
non-branching. Then for each geodesic $t\mapsto\mu_{t}$ in $\mathcal{P}_{p}(M)$
connecting $\mu_{0},\mu_{1}\ll\m$ and $t\in(0,1)$ the geodesics
$s\mapsto\mu_{st}$ and $s\mapsto\mu_{t+s(1-t)}$ are the unique geodesics
connecting $\mu_{0}$ and $\mu_{t}$ and respectively $\mu_{t}$ and
$\mu_{1}$. Furthermore, the (unique) $p$-optimal couplings of $(\mu_{t},\mu_{0})$
and $(\mu_{t},\mu_{1})$ are induced by transport maps $T_{t,0},T_{t,1}:M\to M$.
\end{cor}

\section{Spaces with good transport behavior}

In this section we study spaces where the existence of transport maps
is a priori assumed whenever the initial measure is absolutely continuous.
It turns out that such spaces are already $p$-essentially non-branching. 
\begin{defn}
[Good transport behavior]\label{def:GTB} A metric measure space
$(M,\sfd,\m)$ has \emph{good transport behavior} $\GTB_{p}$ for
$p\in(1,\infty)$, if for all $\mu,\nu\in\mathcal{P}_{p}(M)$ with
$\mu\ll\m$ any optimal transport plan between $\mu$ and $\nu$ is
induced by a map.
\end{defn}
\begin{rem*}
The condition was used in a recent work by F. Galaz-García, A. Mondino,
G. Sosa and the author \cite{GKMS2017GroupCD} in order to study the
orbit structure of groups acting isometrically on metric measure spaces
with $\GTB_{p}$. In particular, it can be used to exclude isometries
with too large fixed point set, see also \cite[Lemma 4.1]{Sosa2016}.
\end{rem*}
\begin{prop}
\label{prop:spaces-with-GTB}The following spaces have $\GTB_{p}$: 
\begin{enumerate}[label=(\roman*)]
\item Essentially non-branching $\mathsf{MCP}(K,N)$-spaces for $p=2$,
$K\in\mathbb{R}$, and $N\in[1,\infty)$. In particular, this includes,
essentially non-branching $\mathsf{CD}^{*}(K',N')$-spaces, essentially
non-branching $\text{\ensuremath{\mathsf{CD}}}(K,N)$-spaces, and
$\mathsf{RCD^{*}(K',N')}$-spaces, see \cite{GRS2016,CM2016TransMapsMCP}.
\item Non-branching, qualitatively non-degenerate spaces for all $p\in(1,\infty)$,
see \cite{CH2015NBTrans} and Definition \ref{def:qual-non-deg} below.
\item Any (local) doubling measure $\mu$ on $(\mathbb{R}^{n},\|\cdot\|_{\operatorname{Euclid}})$
or more generally on a Riemannian manifold, see \cite{GMc1996OptTrans}.
\end{enumerate}
\end{prop}
The last example shows that there is an abundance of spaces with good
transport behavior. However, we will show that the existence of transport
maps prevents too much branching and excludes therefore normed spaces
whose norm is not strictly convex. Note that the main theorem of this
note extends the list above to $p$-essentially non-branching, qualitatively
non-degenerate spaces, see Theorem \ref{thm:MCPimpliesGTB}.

The first two lemmas were proved in a slightly different form in \cite{GKMS2017GroupCD}.
Recall that for $\Gamma\subset M\times M$ and $s,t\in[0,1]$ we define
$\Gamma_{s,t}:=(e_{s},e_{t})\left((e_{0},e_{1})^{-1}\Gamma\right)$.
\begin{lem}
\label{lem:middle-cp-cyl-mon}Let $\Gamma\subset M\times M$ be a
$c_{p}$-cyclically monotone set. Then for any $s,t\in[0,1]$ the
set $\Gamma_{s,t}$ is $c_{p}$-cyclically monotone.
\end{lem}
\begin{proof}
Choose $(x_{s}^{i},x_{t}^{i})\in\Gamma_{s,t}$, $i=1,\ldots,n$, and
note that there are geodesics $\gamma^{(i)}\in\hat{\Gamma}$, $i=1,\ldots,n$,
with $(\gamma_{s}^{(i)},\gamma_{t}^{(i)})=(x_{s}^{(i)},x_{t}^{(i)}).$
By assumption 
\[
\bigcup_{i=1}^{n}\{(\gamma_{0}^{(i)},\gamma_{1}^{(i)})\}\subset\Gamma
\]
is $c_{p}$-cyclically monotone and hence 
\[
\sigma=\frac{1}{n}\sum\delta_{\gamma^{(i)}}
\]
is a $p$-optimal dynamical coupling. Observe that 
\[
\bigcup_{n=1}^{n}\{(\gamma_{s}^{(i)},\gamma_{t}^{(i)})\}=\supp(e_{s},e_{t})_{*}\sigma
\]
is $c_{p}$-cyclically monotone because $(e_{s},e_{t})_{*}\sigma$
is $p$-optimal. Since $\bigcup_{n=1}^{n}\{(\gamma_{0}^{(i)},\gamma_{1}^{(i)})\}\subset\Gamma_{s,t}$
this shows that $\Gamma_{s,t}$ is $c_{p}$-cylically monotone.
\end{proof}
Recall that for a subset $\Gamma\subset M\times M$ we define for
$x\in M$
\[
\Gamma(x)=\{y\in M\,|\,(x,y)\in\Gamma\}.
\]

\begin{lem}
[{\cite[Lemma 4.5]{GKMS2017GroupCD}}]\label{lem:GTB-superdiff}A
metric measure space $(M,\sfd,\m)$ has $\GTB_{p}$ if and only if
for every $c_{p}$-cyclically monotone $\Gamma$, the set $\Gamma(x)$
contains at most one point for $\m$-almost all $x\in M$.

In particular, if $(M,\sfd,\m)$ has $\GTB_{p}$ then for any closed
$c_{p}$-cyclically monotone set $\Gamma$ and $\m$-almost all $x\in M$
there exists a unique geodesic connecting $x$ and $\Gamma(x)$, whenever
the set $\Gamma(x)$ is non-empty.
\end{lem}
\begin{rem*}
The lemma applies in particular to the $c_{p}$-superdifferential
$\partial^{c_{p}}\varphi$ of $c_{p}$-concave functions $\varphi$.
\end{rem*}
\begin{proof}
The first part follows from the Selection Dichotomy for Sets (Theorem
\ref{thm:selection-dichotomy-set}) and the fact that the support
$\Gamma=\supp\pi$ of a $p$-optimal coupling $\pi$ is $c_{p}$-cyclically
monotone. Indeed, the second possibility of the Selection Dichotomy
applied to $\Gamma$ and $\mu=(p_{1})_{*}\pi$ would imply that there
is a compact set $K\subset p_{1}(\Gamma)$ and two maps $T_{1},T_{2}:K\to M$
with $T_{1}(x)\ne T_{2}(x)$ for $x\in K$ and $\{T_{1}(x),T_{2}(x)\}\subset\Gamma(x)$
for $x\in p_{1}(\Gamma)$ and that for $\mu_{0}=\frac{1}{\mu(K)}\mu\big|_{K}$
the following coupling
\[
\frac{1}{2}\left((\operatorname{id}\times T_{1})\mu_{0}+(\operatorname{id}\times T_{1})\mu_{0}\right)
\]
is $p$-optimal and not induced by a transport map. Therefore, either
of the condition implies that $\mu_{0}$ cannot be absolutely continuous
with respect $\m$.

To prove the last statement suppose $(M,\sfd,\m)$ has $\GTB_{p}$
and observe that by the previous lemma whenever $\gamma$ is a geodesic
connecting $x$ and $y\in\Gamma(x)$ then $\gamma_{t}\in\Gamma_{0,t}(x)$. 

Let $D_{t}=\{x\in M\,|\,\Gamma_{0,t}(x)\ne\varnothing\}$ and note
that $D_{t}\subset D_{t'}$ whenever $0\le t'\le t\le1$. Let $(t_{n})_{n\in\mathbb{N}}$
be dense in $(0,1]$ with $t_{1}=1$ and choose a measurable set $\Omega_{n}\subset D_{1}$
of full $\m$-measure in $D_{1}$ such that $\Gamma_{0,t_{n}}(x)$
is single-valued for all $x\in\Omega_{n}$. Then $\Omega=\cap_{n\in\mathbb{N}}\Omega_{n}$
also has full $\m$-measure in $D_{1}$. Let $\gamma$ and $\eta$
be two geodesics connecting $x\in\Omega$ and $y\in\Gamma_{0,1}(x)$.
If $\gamma$ and $\eta$ were distinct then there is an open interval
$I\subset(0,1)$ such that $\gamma_{s}\ne\eta_{s}$ for all $s\in I$.
In particular, there is an $n>0$ such that $t_{n}\in I$. Hence $\gamma_{t_{n}}\ne\eta_{t_{n}}$
and $\Gamma_{0,t_{n}}(x)$ is not single-valued. However, this is
a contradiction as $x\in\Omega\subset\Omega_{n}$ implies that $\Gamma_{0,t_{n}}(x)$
is single-valued.
\end{proof}
\begin{lem}
\label{lem:from-mid-to-superdiff}Let $\varphi$ be a $c_{p}$-concave
function and $(x_{0},x_{1}),(y_{0},y_{1})\in\partial^{c_{p}}\varphi$
be such that for some $t_{0}\in(0,1)$ it holds that $x_{t_{0}}=y_{t_{0}}$,
where $x_{t}$ and $y_{t}$ are $t$-midpoints of $(x_{0},x_{1})$
and $(y_{0},y_{1})$ respectively. Then $(x_{0},y_{1}),(y_{0},x_{1})\in\partial^{c_{p}}\varphi$. 
\end{lem}
\begin{proof}
Choose geodesics $s\mapsto x_{s}$ and $s\mapsto y_{s}$ between $x_{0}$
and $x_{1}$ and resp. $y_{0}$ and $y_{1}$ and define 
\[
\mu_{s}=\frac{1}{2}\left(\delta_{x_{s}}+\delta_{y_{s}}\right).
\]
Note that $(\varphi,\varphi^{c_{p}})$ is a dual solution for the
measures $\mu_{0}$ and $\mu_{1}$. 

We write $\varphi_{t}=t^{p-1}\varphi$ and note that the function
$\varphi_{t}$ is $c$-concave and $(\varphi_{t},\varphi_{t}^{c_{p}})$
a dual solution for the measure $\mu_{0}$ and $\mu_{t}$ (\cite[2.9 and Remark after 2.1]{Kell2015}).
Denote the $c_{p}$-duals of $\varphi$ and $\varphi_{t}$ by $\psi$
and $\psi_{t}$ respectively. 

Since $\partial^{c_{p}}\varphi$ is $c_{p}$-cyclically monotone,
Lemma \ref{lem:cyc-mon-intersect-implies-same-length} shows that
\[
d(x_{0},x_{1})=d(y_{0},y_{1})=d(x_{0},y_{1})=\sfd(y_{0},x_{1}).
\]
Furthermore, $(x_{0},x_{t}),(y_{0},y_{t})\in\partial^{c_{p}}\varphi_{t}$
by the choice of geodesics $s\mapsto x_{s}$ and $s\mapsto y_{s}$.
Hence 
\[
\varphi_{t}(x_{0})+\psi_{t}(x_{t})=\sfd^{p}(x_{0},x_{t})=\sfd^{p}(y_{0},y_{t})=\varphi_{t}(y_{0})+\psi_{t}(y_{t}).
\]
For $t=1$ we obtain 
\[
\varphi(x_{0})+\psi(x_{1})=\varphi(y_{0})+\psi(y_{1}),
\]
and because $x_{t_{0}}=y_{t_{0}}$ for some $t_{0}\in(0,1)$, we also
have 
\[
\varphi_{t_{0}}(x_{0})+\psi_{t_{0}}(x_{t})=\varphi_{t_{0}}(y_{0})+\psi_{t_{0}}(x_{t})
\]
implying 
\[
\varphi(x_{0})=t_{0}^{1-p}\varphi_{t_{0}}(x_{0})=t_{0}^{1-p}\varphi_{t_{0}}(y_{0})=\varphi(y_{0}).
\]
Therefore,
\[
\varphi(x_{0})+\psi(y_{1})=\sfd^{p}(x_{0},x_{1})=\sfd^{p}(x_{0},y_{1})
\]
which shows $(x_{0},y_{1})\in\partial^{c_{p}}\varphi$. Similarly,
it holds $(y_{0},x_{1})\in\partial^{c_{p}}\varphi$. 
\end{proof}
The following is a direct application of the last lemma. 
\begin{prop}
\label{prop:GTB-ENB}Assume $(M,\sfd,\m)$ has $\GTB_{p}$ and let
$\mu_{0},\mu_{1}\in\mathcal{P}^{p}(M)$ with $\mu_{0}\ll\m$. Then
there is a unique $p$-optimal dynamical plan $\sigma\in\mathrm{OptGeo}_{p}(\mu_{0},\mu_{1})$
and a measurable set $\mathsf{L}$ with $\sigma(\mathsf{L})=1$ which
is non-branching to the right. In particular, a metric measure space
with $\GTB_{p}$ is $p$-essentially non-branching. 
\end{prop}
\begin{proof}
Let $\mu_{0},\mu_{1},$ and $\pi$ be as above and $T$ be a $p$-optimal
transport map between $\mu_{0}$ and $\mu_{1}$. Assume $\mathcal{T}:M\times M\to\Geo(M,d)$
is a measurable selection such that $\mathcal{T}(x,y)_{0}=x$ and
$\mathcal{T}(x,y)_{0}=y$ and define a measurable map $\mathcal{S}:M\to\Geo(M,d)$
by 
\[
\mathcal{S}(x)=\mathcal{T}(x,T(x)).
\]
Note if $\varphi$ is a dual solution then $\supp\pi\subset\partial^{c_{p}}\varphi$. 

By Lemma \ref{lem:GTB-superdiff} there a Borel set $A$ of full $\mu_{0}$-measure
such that 
\[
\left(A\times M\right)\cap\partial^{c_{p}}\varphi=\left(A\times M\right)\cap\operatorname{graph}T
\]
and for all $x\in A$ the geodesic $\mathcal{S}(x)$ is the unique
geodesic connecting $x$ and $T(x)$. This implies immediately that
the dynamical coupling $\sigma=\mathcal{S}_{*}\mu_{0}$ is the unique
$p$-optimal dynamical coupling between $\mu_{0}$ and $\mu_{1}$. 

It suffices to show that $\mathsf{L}=\mathcal{S}(A)$ is non-branching
to the right. For this let $\gamma,\eta\in\mathsf{L}$ be two geodesics
with $\ell(\gamma)=\ell(\eta)$ with $\gamma_{t}=\eta_{t}$ for some
$t\in(0,1)$. Lemma \ref{lem:from-mid-to-superdiff} implies that
$(\gamma_{0},\eta_{1})$ and $(\eta_{0},\gamma_{1})$ are both in
$\partial^{c_{p}}\varphi$. However, since $\gamma_{0},\eta_{0}\in A$
this means $\gamma_{1}=\eta_{1}=T(\gamma_{0})$. If we define now
\[
\tilde{\gamma}_{s}=\begin{cases}
\gamma_{s} & s\in[0,t]\\
\eta_{s} & s\in[t,1]
\end{cases}
\]
then $\tilde{\gamma}_{s}$ is also a geodesic connecting $\gamma_{0}$
and $\gamma_{1}$. The choice of $A$ yields $\tilde{\gamma}\equiv\gamma$.
Thus $\gamma_{s}=\eta_{s}$ for $s\in[t,1]$ implying that $\mathsf{L}$
is non-branching to the right.
\end{proof}
\begin{rem*}
The conclusion in the first part of Proposition \ref{prop:GTB-ENB}
above is stronger than the ordinary $p$-essentially non-branching
property as it takes into account arbitrary final measures rather
than just absolutely continuous ones. 
\end{rem*}
\begin{cor}
Assume $(M,d,\m)$ has good transport behavior $\GTB_{p}$ and is
strongly non-degenerate $\sND_{p}$ (see Definition \ref{def:str-non-deg}
below). Then for any $\mu_{0},\mu_{1}\in\mathcal{P}_{p}(M)$ with
$\mu_{0}\ll\m$ there is a unique $p$-optimal dynamical coupling
$\sigma$ between $\mu_{0}$ and $\mu_{1}$ and this coupling is concentrated
on a set of non-branching geodesics. Furthermore, $(e_{t})_{*}\sigma\ll\m$
for all $t\in(0,1)$. In particular, $\m$ has the strong interpolation
property $\sIP_{p}$.
\end{cor}

\subsection*{Example of essentially non-branching spaces with bad geometric behavior}

In this section we construct a measure on the tripod that is essentially
non-branching and for any two absolutely continuous measures there
is a unique transport map. However, the obvious branching in the tripod
shows that there is no measure that makes the tripod into a space
with good transport behavior. 
\begin{defn}
A metric measure space $(M,d,\m)$ has the \emph{weak good transport
behavior} $\GTB_{w,p}$ for $p\in(1,\infty)$, if for all $\mu_{0},\mu_{1}\in\mathcal{P}_{p}(M)$
with $\mu_{0},\mu_{1}\ll\m$ any optimal transport plan between $\mu_{0}$
and $\mu_{1}$ is induced by a map. 

Let $(\mathsf{T},d)$ be the tripod, i.e. $\mathsf{T}$ is obtained
by gluing together three intervals $I_{i}=[0_{i},1_{i}]$, $i=1,2,3$
at $\mathbf{0}=0_{1}=0_{2}=0_{3}$ and $d$ is the corresponding length
metric. Denote by $T_{i}$ the natural inclusions $[0,1]\to I_{i}\subset\mathsf{T}$.
\end{defn}
\begin{example*}
There is continuum of measures $\m$ of full support on $\mathsf{T}$
such that $(\mathsf{T},d,\m)$ is $p$-essentially non-branching and
has the weak good transport behavior $\GTB_{w,p}$ for all $p\in(1,\infty)$.
\end{example*}
\begin{proof}
[Sketch of the construction]Let $\nu_{0},\nu_{1},\nu_{2}$ three
non-atomic probability measures on $[0,1]$ with full support and
$\Omega_{i}$ three disjoint sets such that $\mu_{i}(\Omega_{j})=\delta_{ij}$.
Define a measure $\m$ on $\mathsf{T}$ by 
\[
\m\big|_{I_{i}}=(T_{i})_{*}\nu_{i}
\]
and a set $\Omega=\cup T_{i}(\Omega_{i}).$ Note that $\m(M\backslash\Omega)=0.$ 

If $x,y\in\Omega$ satisfy $d(\mathbf{0},x)=d(\mathbf{0},y)$ then
$x,y\in\Omega_{i}$ for exactly one $i=1,2,3$ and $x=y$. Thus since
branching can only happen at $\mathbf{0}$, any two geodesics with
endpoints in a $c_{p}$-cyclically monotone set $\Gamma\subset\Omega\times\Omega$
which intersect at a point $t\in(0,1)$ must be equal. In particular,
any such $\Gamma$ is already non-branching. Note that whenever $\mu_{0},\mu_{1}\ll\m$
and $\pi$ is a $p$-optimal coupling then 
\[
\pi(\Omega\times\Omega)=1
\]
so that $\pi$ is concentrated on the non-branching set $\supp\pi\cap(\Omega\times\Omega)$.

To obtain transport maps it is sufficient to assume $\mu_{0}\ll\m\big|_{I_{i}}$.
In that case let $S_{i}:\mathsf{T}\to[-1,1]$ be the map that collapses
$I_{j}$ and $I_{k}$ for $i\ne k,j$ where we assume $I_{i}$ corresponds
to $[-1,0]$. Note also that $S_{i}$ restricted to $\Omega$ is invertible
hence $(S_{i})_{*}\m$ is a non-atomic measure. This can be used to
show that the (unique) $p$-optimal transport map between $(S_{i})_{*}\mu_{0}$
and $(S_{i})_{*}\mu_{1}$ can be pulled back to a $p$-optimal transport
map. 

This construction works more general for all cost function $h(d(\cdot,\cdot))$
with $h$ strictly convex and increasing. 
\end{proof}

\section{Existence of absolutely continuous interpolations for non-degenerate
measures}

In this section we prove the existence of absolutely continuous interpolation
measures if the initial measure is absolutely continuous and the background
measure satisfies certain non-degenericity conditions. In order to
avoid proving very similar results for final measures supported on
finite sets and then on general sets, we generalize the construction
to optimal couplings concentrated on so called non-degenerate sets.

\subsection*{Non-degenerate measures and sets}

The following condition was introduced in \cite{CM2016MeasRigid}
and is based on stronger variant called \emph{qualitative non-degenericty}
(see blow) introduced earlier in \cite{CH2015NBTrans}. 
\begin{defn}
[non-degenerate measure]A metric measure space $(M,d,\m)$ is called
\emph{non-degenerate} if for all Borel sets $A$ with $\m(A)>0$ it
holds $\m(A_{t,x})>0$ for $t\in(0,1)$. 
\end{defn}
Our main goal is to prove existence of absolutely continuous interpolations.
We formalize the general a priori existence by the following condition.
\begin{defn}
[interpolation property]A metric measure space $(M,d,\m)$ is said
to have the \emph{interpolation property} $\IP_{p}$ for some $p\in(1,\infty)$
if for all $\mu_{0},\mu_{1}\in\mathcal{P}_{p}(M)$ with $\mu_{0}\ll\m$,
all $p$-optimal couplings $\pi\in\Opt_{p}(\mu_{0},\mu_{1})$ and
all $t\in(0,1)$ there is a $p$-optimal dynamical coupling $\sigma$
between $\mu_{0}$ and $\mu_{1}$ with $(e_{t})_{*}\sigma\ll\m$. 

It has the \emph{strong interpolation property} $\sIP_{p}$ for all
$p$-optimal dynamical coupling $\sigma$ between $\mu_{0}$ and $\mu_{1}$
it holds $(e_{t})_{*}\sigma\ll\m$.
\end{defn}
In order to show that the interpolation property $\IP_{p}$ holds
we will study the supports of optimal couplings and need a non-degenericity
condition of sets $\Gamma\subset M\times M$. 

The following notation will be used: Given a set $A$ and $s\in[0,1]$
define the set $\Gamma^{A,s}$ by 
\[
((e_{0},e_{1})e_{s}^{-1}(A))\cap\Gamma,
\]
i.e. we throw out all endpoints which cannot be reached via geodesics
having an $s$-midpoint in $A$. As above $\Gamma_{t}^{A,s}$ equals
$e_{t}(\Gamma^{A,s})$ whenever $t\in[0,1]$. Observe that if $A$
is analytic then $\Gamma_{t}^{A,s}$ is analytic for all $s,t\in[0,1]$.
Also in case $s=t=0$ this simplifies to $\Gamma_{0}^{A,0}=\Gamma_{0}\cap A$.
\begin{defn}
[non-degenerate set]An Borel  set $\Gamma\subset M\times M$ is \emph{non-degenerate}
(with respect to $\m$) if for all Borel sets $A$ with $\m(\Gamma_{0}\cap A)>0$
it holds $\m(\Gamma_{t}^{A,0})>0$ whenever $t\in(0,1)$.
\end{defn}
It is easy to see that $B\times\{x\}$ is non-degenerate for all $x\in M$
and all Borel sets $B\subset M$ whenever $\m$ is non-degenerate. 
\begin{defn}
[strong non-degenerate measure]\label{def:str-non-deg}A metric measure
space $(M,d,\m)$ is \emph{strongly non-degenerate $\sND_{p}$ }for
some $p\in(1,\infty)$ if every $c_{p}$-cyclically monotone Borel
set $\Gamma$ is non-degenerate. 
\end{defn}
\begin{rem*}
By abuse of notation we say $\m$ is (strongly) non-degenerate or
has the (strong) interpolation property if $(M,d,\m)$ is (resp. has)
the corresponding property.
\end{rem*}
It is easy to see that any strongly non-degenerate measure $\m$ is
also non-degenerate. Furthermore, a measure with strong interpolation
property $\sIP_{p}$ is necessarily strongly non-degenerate $\sND_{p}$.
The converse is true if the $p$-optimal dynamical coupling $\sigma$
is unique. In general this is wrong as can be seen by the metric measure
space $(\mathbb{R}^{n},\|\cdot-\cdot\|_{\infty},\lambda^{n})$ which
has many non-absolutely continuous interpolations from the Lebesgue
measures restricted to the unit ball to the delta measure at the origin,
see also Remark after the proof of Lemma \ref{lem:p-ess-nb-implies-p-str-non-deg}. 

However, the interpolation property is sufficient to show that the
space is strong non-degenerate. Via the existence of absolutely continuous
interpolations in the next section one can show that both properties
are actually equivalent.
\begin{lem}
Assume $(M,d,\m)$ is a metric measure space having the interpolation
property $\IP_{p}$. Then $(M,d,\m)$ is strongly non-degenerate $\sND_{p}$.
\end{lem}
\begin{proof}
Let $A$ be a Borel set  and $\Gamma$ be a $c_{p}$-cyclically monotone
Borel set with $\m(p_{1}(\Gamma)\cap A)>0$. Without loss of generality
$A\subset p_{1}(\Gamma)=\Gamma_{0}$. Let $\mu_{0}=\frac{1}{\m(A)}\m\big|_{A}$
and choose a measurable selection $T$ of $\Gamma\cap(A\times M)$.
Then $\pi=(\operatorname{id}\times T)_{*}\mu_{0}$ is a $p$-optimal
coupling. Let $\sigma$ be given by the interpolation property. Then
$(e_{t})_{*}\sigma(\Gamma_{t}^{A,0})=1$ and $(e_{t})_{*}\sigma\ll\m$
implying $\m(\Gamma_{t}^{A,0})>0$. Because $A$ and $\Gamma$ are
arbitrary we conclude that $(M,d,\m)$ ist strongly non-degenerate
$\sND_{p}$.
\end{proof}

\subsection*{The GKS-Construction}

In this section we construct an absolutely continuous interpolation
$\mu_{t}$ given a $p$-optimal coupling $\pi$ which is concentrated
(in a consistent way) on a non-degenerate set $\Gamma$ and its first
marginal $(p_{1})_{*}\pi$ is absolutely continuous. Furthermore,
we find a Borel set $\tilde{\Gamma}\subset\Gamma$ of full $\pi$-measure,
such that $\mu_{t}$ ``sees'' all points in $\tilde{\Gamma}_{t}$
of positive $\m$-measure. The last property turns out to be crucial
in order to apply the idea of Cavalletti\textendash Huesmann \cite{CH2015NBTrans}
in the setting of essentially non-branching spaces, see proof of Lemma
\ref{lem:no-overlap-non-deg} and Theorem \ref{thm:MCPimpliesGTB}. 

The proof of Theorem \ref{thm:abs-cts-interpolation} below is based
on the following generalized form of the Lebesgue decomposition which
can be found in \cite[Section 9.4]{Rudin2008}. One part of the result
was proven by Glicksberg and the other by König and Sievers owing
the name \emph{GKS-Decomposition}, see \cite[9.4.1]{Rudin2008}.
\begin{lem}
[GKS-Decomposition {\cite[9.4.4]{Rudin2008}}] Let $(M,d)$ be a locally
compact complete separable metric space and $\mathcal{B}\subset\mathcal{P}(M)$
be a weakly compact and linearly convex subset of probability measures.
Then every non-negative finite measure $\tilde{\m}$ has a unique
decomposition
\[
\tilde{\m}=\tilde{\m}_{a}+\tilde{\m}_{s}
\]
such that $\tilde{\m}_{a}\ll\mu$ for some $\mu\in\mathcal{B}$ and
there is a Borel set $F$ which is a countable union of closed subsets
such that $\m_{s}$ is concentrated on $F$ and, in addition, $F$
is $\mathcal{B}$-null, i.e. it holds $\tilde{\m}_{s}(M\backslash F)=0$
and $\nu(F)=0$ for all $\nu\in\mathcal{B}$.
\end{lem}
\begin{rem*}
(1) Linearly convex of $\mathcal{B}$ means that whenever $\mu,\nu\in\mathcal{B}$
then also $(1-\lambda)\mu+\lambda\nu\in\mathcal{B}$ for all $\lambda\in[0,1]$. 

(2) The lemma is usually stated for compact Hausdorff spaces. However,
one can embed $M$ into the one-point-compactification $M^{*}=\{*\}\cup M$
such that $\mathcal{B}$ is still compact in $\mathcal{P}(M^{*})$.
Note that $M^{*}$ is a compact Hausdorff space. Since each of the
involved measures gives zero measure to the set $\{*\}$, we see that
the lemma also holds for general locally compact Hausdorff spaces.
In particular, it holds for proper metric spaces.

(3) Recall that $\m$ is a locally bounded measure if $(M,d,\m)$
is a proper metric measure space. In that case there is a continuous
function $\varphi:[0,\infty)\to(0,1]$ such that $\tilde{\m}=\varphi(d(x_{0},\cdot))\m$
is a finite measure. Then the unique decomposition of $\m$ with respect
to $\mathcal{B}$ is given by $\m=(\varphi(d(x_{0},\cdot))^{-1}\tilde{\m}_{a}+(\varphi(d(x_{0},\cdot))^{-1}\tilde{\m}_{s}$. 
\end{rem*}
Before stating the main theorem of this section we need the following
technical lemmas. 
\begin{lem}
\label{lem:dyn-cpl-pass-through-A}If $A$ is an analytic set and
$\pi$ is a coupling concentrated on $\Gamma^{A,t}$ then there is
a dynamical coupling $\sigma$ concentrated on $e_{t}^{-1}(A)\cap\hat{\Gamma}$.
In particular, $(e_{t})_{*}\sigma(A)=1$.
\end{lem}
\begin{proof}
Since $A$ is analytic, the set
\begin{align*}
\Lambda & =\{(\gamma_{0},\gamma_{1},\gamma)\in\Gamma\times\Geo_{[0,1]}(M,d)\,|\,\gamma\in e_{t}^{-1}(A)\cap\hat{\Gamma}\}\\
 & =(e_{0},e_{1},\operatorname{id})\left(e_{t}^{-1}(A)\cap(e_{0},e_{1})^{-1}(\Gamma)\right)
\end{align*}
is also analytic. Thus by von Neumann's Measurable Selection Theorem
there is a selection $S:\Gamma\to\Geo_{[0,1]}(M,d)$ such that $(x,y,S(x,y))\in\Lambda$
for all $(x,y)\in\Gamma$. In particular, $S(x,y)_{t}\in A$ . To
conclude just observe that $\sigma=S_{*}\pi$ is concentrated on $e_{t}^{-1}(A)\cap\hat{\Gamma}$.
\end{proof}
\begin{lem}
\label{lem:cpt-t-midpoints}Assume $(M,d)$ is a proper geodesic space.
Then for all measures $\mu_{0},\mu_{1}\in\mathcal{P}_{p}(M)$ and
every $p$-optimal coupling $\pi$ the following set of $t$-midpoints
\[
\mathcal{B}=\{(e_{t})_{*}\sigma\,|\,\sigma\in\OptGeo(\mu_{0},\mu_{1}),(e_{0},e_{1})_{*}\sigma=\pi\}
\]
is linearly convex, compact in $\mathcal{P}_{p}(M)$ and weakly compact
in $\mathcal{P}(M)$.
\end{lem}
\begin{rem*}
Similar arguments also show that for finitely many $\{s_{1},\ldots,s_{n}\}\subset[0,1]$
and a measure $\boldsymbol{\mathcal{\pi}}\in\mathcal{P}(M^{n})$ the
set
\[
\mathcal{C}=\{(e_{t})_{*}\sigma\,|\,\sigma\in\OptGeo(\mu_{0},\mu_{1}),(e_{s_{1}},\ldots e_{s_{n}})_{*}\sigma=\boldsymbol{\pi}\}
\]
 is linearly convex, compact in $\mathcal{P}_{p}(M)$ and weakly compact
in $\mathcal{P}(M)$.
\end{rem*}
Note that the result shows that the GKS-Decomposition can be applied
to the set $\mathcal{B}$. Its proof is given at the end of this section. 

In order to make the main theorem more readable we introduce the following
condition. It won't be used anywhere else but here.
\begin{defn}
A coupling $\pi\in\mathcal{P}(M\times M)$ is \emph{strongly consistent}
if for all $\tilde{\pi}\ll\pi$ and every measurable set $\Gamma'$
with $\tilde{\pi}(\Gamma')=1$ there is a non-degenerate, $c_{p}$-cyclically
monotone Borel set $\Gamma\subset\Gamma'$ with $\tilde{\pi}(\Gamma)=1$.
\end{defn}
Note that whenever $\pi$ is strongly consistent then any coupling
$\pi'$ with $\pi'\ll\pi$ is strongly consistent as well.
\begin{thm}
[GKS-Construction]\label{thm:abs-cts-interpolation}Let $(M,d,\m)$
be a proper metric measure space. Assume $\pi$ is a strongly consistent,
$p$-optimal coupling between $\mu_{0}\ll\m$ and $\mu_{1}$. Then
for every $t\in(0,1)$ there is a $p$-optimal dynamical coupling
$\sigma$ such that $(e_{0},e_{1})_{*}\sigma=\pi$ and 
\begin{align*}
\mu_{t}=(e_{t})_{*}\sigma & \ll\m.
\end{align*}
Furthermore, $\mu_{t}$ is maximal in the following sense: Let $\Gamma$
be a Borel set of full $\pi$-measure and 
\[
\m\big|_{\Gamma_{t}}=g\mu_{t}+\m\big|_{F}
\]
be the the Lebesgue decomposition of $\m\big|_{\Gamma_{t}}$ with
respect to $\mu_{t}$ where $F\subset\Gamma_{t}$ is a Borel with
$\mu_{t}(F)=0$. Then $\pi$ is concentrated on a Borel set $\tilde{\Gamma}\subset\Gamma\backslash\Gamma^{F,t}$
and it holds 
\[
\m\big|_{\tilde{\Gamma}_{t}}\ll\mu_{t}.
\]
\end{thm}
\begin{rem*}
Absolute continuity and maximality imply $\m\big|_{\tilde{\Gamma}_{t}}\ll\mu_{t}\ll\m\big|_{\tilde{\Gamma}_{t}}$
as $\pi$ is concentrated on $\tilde{\Gamma}$.
\end{rem*}
\begin{cor}
\label{cor:abs-cts-interpolation}Suppose $\mu_{t}\ll\m$ and $\tilde{\Gamma}$
are constructed from $\pi$ as above. If there is a strongly consistent
coupling $\pi_{t,1}$ of $\mu_{t}$ and $\mu_{1}$ then for each $s\in(t,1)$
there is a $p$-optimal dynamical coupling $\hat{\sigma}$ such that
$(e_{0},e_{1})_{*}\hat{\sigma}=\pi$, $(e_{t})_{*}\hat{\sigma}=\mu_{t}$
and $(e_{s})_{*}\hat{\sigma}\ll\m$.
\end{cor}
\begin{proof}
[Proof of the theorem]Let $\mathcal{B}$ be defined as in Lemma \ref{lem:cpt-t-midpoints}
above. We split the proof into two steps.

\textbf{\noun{Step $\mathbf{1}$}}: \texttt{There is a $\mu_{t}\in\mathcal{B}$
which is maximal in the sense of the theorem and $\rho_{t}\ne0$ where
$\mu_{t}=\rho_{t}\m+\mu_{t}^{s}$ is the Lebesgue decomposition of
$\mu_{t}$ with respect to $\m$.}

Let $\Gamma\subset\supp\pi$ be a non-degenerate, $c_{p}$-cyclically
monotone Borel set with $\pi(\Gamma)=1$. Since $\mathcal{B}$ is
weakly compact and linearly convex we can apply the GKS-Decomposition
to $\m\big|_{\Gamma_{t}}$ and obtain a measure $\mu_{t}\in\mathcal{B}$
such that 
\[
\m\big|_{\Gamma_{t}}=g\mu_{t}+\m_{s}
\]
and there is a Borel set $F\subset\Gamma_{t}$ such that $\tilde{\mu}_{t}(F)=0$
for all $\tilde{\mu}_{t}\in\mathcal{B}$ and $\m_{s}(M\backslash F)=0$.
Thus $\m_{s}=\m\big|_{F}$ and $g(x)>0$ for $\mu_{t}$-almost all
$x\in M$. 

We claim that $\pi$ is concentrated on $\Gamma\backslash\Gamma^{F,t}$.
Assume, by contradiction, that 
\[
\lambda=\pi(\Gamma^{F,t})>0
\]
Then for $f=\chi_{\Gamma^{F,t}}$ the coupling$\sigma_{f}$ is a $p$-optimal
dynamical coupling between $(e_{0})_{*}\sigma$ and $(e_{1})_{*}\sigma$
and $\pi_{F}=(e_{0},e_{1})_{*}\sigma$ is concentrated on $\Gamma^{F,t}$.
Note that $\pi=\lambda\pi_{F}+(1-\lambda)\check{\pi}$ where $\check{\pi}=(e_{0},e_{1})_{*}\sigma_{1-f}$
for some $p$-optimal dynamical coupling $\sigma$ with $(e_{0},e_{1})_{*}\sigma=\pi$
and $(e_{t})_{*}\sigma=\mu_{t}$.

By Lemma \ref{lem:dyn-cpl-pass-through-A} there is a $p$-optimal
dynamical coupling $\tilde{\sigma}_{f}$ induced by $\pi_{F}$ which
concentrated on $\hat{\Gamma}^{F,t}$ such that $(e_{t})_{*}\sigma(F)=1$.
However, by Lemma \ref{lem:restriction-property} the dynamical coupling
\[
\tilde{\sigma}=\lambda\tilde{\sigma}_{f}+(1-\lambda)\sigma_{1-f}
\]
is also $p$-optimal with 
\[
(e_{0},e_{1})_{*}\tilde{\sigma}=\lambda\pi_{F}+(1-\lambda)\tilde{\pi}=\pi.
\]
Thus $\tilde{\mu}_{t}=(e_{t})_{*}\tilde{\sigma}\in\mathcal{B}$ with
\[
\tilde{\mu}_{t}(F)\ge\lambda\mu_{t}^{f}(F)=\lambda>0
\]
contradicting the properties of GKS-Decomposition. Hence $\pi(\Gamma\backslash\Gamma^{F,t})=1$. 

Let 
\begin{align*}
\mu_{t} & =\rho_{t}\m+\mu_{t}^{s}
\end{align*}
be the Lebesgue decomposition of $\mu_{t}$ with respect to $\m$
with $\mu_{t}^{s}\perp\m$. Note that $\rho_{t}(x)>0$ for $\m$-almost
all $x\in A_{t}=(\Gamma\backslash\Gamma^{F,t})_{t}$. By assumption
$\pi$ is concentrated on a non-degenerate $c_{p}$-cyclically monotone
Borel set $\tilde{\Gamma}\subset\Gamma\backslash\Gamma^{F,t}$. Since
$\mu_{0}(\tilde{\Gamma}_{0})=1$ and $\mu_{0}\ll\m$ it holds $\m(\tilde{\Gamma}_{0})>0$
and thus $\m(\tilde{\Gamma}_{t})>0$ by non-degenericity of $\tilde{\Gamma}$
implying $\rho_{t}\ne0$, i.e. the absolutely continuous part of $\mu_{t}$
is non-trivial. Finally observe that $\rho_{t}(x)>0$ for $\m$-almost
all $x\in A_{t}$ shows that $\m(A_{t}\backslash\tilde{\Gamma}_{t})=0$
hence $\m\big|_{\tilde{\Gamma}_{t}}\ll\mu_{t}$ yields maximality
of $\mu_{t}$.

\textbf{\noun{Step $\mathbf{2}$}}: \texttt{Given $\mu_{t}=\rho_{t}\m+\mu_{t}^{s}$
and $\tilde{\Gamma}$ as in }\texttt{\noun{Step 1}}\texttt{, there
is a $\mu_{t}^{*}\in\mathcal{B}$ with $\mu_{t}^{*}=\rho_{t}^{*}\m$
and $\rho_{t}\le\rho_{t}^{*}$, and $\mu_{t}^{*}$ is maximal in the
sense of the theorem.}

We define a partial ordering on subsets of $\mathcal{B}$ and show
that maximal elements exist and are absolutely continuous: For $\rho\in L_{\ge0}^{1}(\m)$
with $\int\rho d\m\in[0,1]$ set 
\[
\mathcal{B}_{\rho}=\{\mu\in\mathcal{B}\,|\,\mu=\rho\m+\mu^{s}\}
\]
where $\mu=\rho\m+\mu^{s}$ is the Lebesgue decomposition of $\mu$
with respect to $\m$. Note that we identify two $L^{1}(\m)$-functions
which agree $\m$-almost everywhere.

Let 
\[
\mathbb{K}=\{\rho\in L_{\ge0}^{1}(\m)\,|\,\int\rho d\m\in[0,1],\mathcal{B}_{\rho}\ne\varnothing\}
\]
and write 
\[
\rho'\succ\rho\;:\Longleftrightarrow\rho'\ge\rho,\rho'\ne\rho.
\]
 This is a partial ordering of $\mathbb{K}$. Also note that 
\[
\bigcup_{\rho\in\mathbb{K}}\mathcal{B}_{\rho}
\]
is a partition of $\mathcal{B}$. Hence the partial order $\succ$
induces one on this partition.

Assume $\mu\in\mathcal{B}_{\rho}$ is not absolutely continuous with
respect to $\m$. Then the decomposition $\mu=\rho\m+\mu^{s}$ induces
a decomposition of $\pi$ as follows
\[
\pi=(1-\lambda)\pi_{a}+\lambda\pi_{s}
\]
where $\pi_{a}$ and $\pi_{s}$ are uniquely defined measures in $\mathcal{P}(M\times M)$
with 
\[
(e_{t})_{*}(1-\lambda)\pi_{a}=\rho\m
\]
and 
\[
(e_{t})_{*}\lambda\pi_{s}=\mu^{s}\ne0.
\]
Let 
\[
\tilde{\mu}_{i}=(e_{i})_{*}\pi_{s}\quad i=0,1.
\]
Since $\pi_{s}\ll\pi$, we see that $\pi_{s}$ is a strongly consistent
$p$-optimal coupling between $\tilde{\mu}_{0}\ll\m$ and $\tilde{\mu}_{1}$.
Thus \noun{Step $1$} above is applicable to $(\tilde{\mu}_{0},\tilde{\mu}_{1})$
and there is a $t$-midpoint $\tilde{\mu}_{t}$
\[
\tilde{\mu}_{t}=\tilde{\rho}\m+\tilde{\mu}^{s}
\]
with $\tilde{\rho}\ne0$ such that 
\[
\mu_{t}^{'}=(\rho+\lambda\tilde{\rho})\m+\lambda\tilde{\mu}^{s}
\]
is still in $\mathcal{B}$. Hence $\rho+\lambda\tilde{\rho}\in\mathbb{K}$
and $\rho+\lambda\tilde{\rho}\succ\rho$. In particular, any $\rho\in\mathbb{K}$
with $\int\rho d\m\ne1$ is not maximal with respect the partial order
$\succ$. Also note that any element $\rho\in\mathbb{K}$ satisfying
$\int\rho d\m=1$ is automatically maximal. To finish the proof it
suffices to show that there are maximal elements above any $\rho\in\mathbb{K}$.

For this we want to apply Zorn's Lemma: Let $\{\rho_{i}\}_{i\in I}$
be a totally ordered chain where $I$ is a totally ordered index set.
Then choose $\mu_{i}\in\mathcal{B}_{\rho_{i}}$ and observe by compactness
of $\mathcal{B}$ there is a subnet $I'\subset I$ such that $\lim_{i\in I'}\mu_{i}=\mu\in\mathcal{B}$.
Then the net $(\rho_{i})_{i\in I'}$ is an increasing family of non-negative
$L^{1}(\m)$-function so that by monotone convergence there is an
$\rho\in L^{1}(\m)$ such that 
\[
\rho=\lim_{i\in I'}\rho_{i}
\]
and $\int\rho d\m\in[0,1]$. Since 
\[
\rho_{i}\m\le\mu_{j}
\]
whenever $i\le j$, it holds $\rho_{i}\m\le\mu$ and thus $\rho\m\le\mu$
where $\mu\le\nu$ means $\mu(A)\le\nu(A)$ for all Borel sets $A$.
In particular, the Lebesgue decomposition of $\mu$ is given by 
\[
\mu=\hat{\rho}\m+\mu^{s}
\]
for some $\hat{\rho}\in\mathbb{K}$ with $\hat{\rho}\succeq\rho\succeq\rho_{i}$,
i.e. the chain $\{\rho_{i}\}_{i\in I}$ has a maximal element in $\mathbb{K}$.
Therefore, Zorn's Lemma applied to $(\mathbb{K},\succeq)$ gives the
existence of at least one maximal element $\rho_{t}^{*}\in\mathbb{K}$
with $\rho_{t}^{*}\ge\rho_{t}$. Choosing $\mu_{t}^{*}=\rho_{t}^{*}\m$
gives a measure satisfying the statements of the theorem. 

Finally, the properties of $\mu_{t}$ and $\tilde{\Gamma}$ imply
$\m\big|_{\tilde{\Gamma}_{t}}\ll\rho_{t}\m\ll\mu_{t}^{*}\ll\m\big|_{\tilde{\Gamma}_{t}}$.
Thus for any other Borel set $\Gamma$ with $\pi(\Gamma)=1$ there
is a Borel set $\Gamma^{'}\subset\Gamma\cap\tilde{\Gamma}$ with $\pi(\Gamma^{'})=1$.
But then $\m\big|_{\tilde{\Gamma}_{t}}\ll\mu_{t}^{*}\ll\m\big|_{\Gamma_{t}^{'}}$
implying $\m(\tilde{\Gamma}_{t}\backslash\Gamma_{t}^{'})=0$. This
yields immediately $\m\big|_{\bar{\Gamma}_{t}}\ll\mu_{t}^{*}$ and
thus maximality of $\mu_{t}^{*}$.
\end{proof}
In order to apply the theorem we need to prove that we find $p$-optimal
couplings satisfying the assumptions of the theorem. 
\begin{lem}
\label{lem:finite-non-deg-implies-non-deg-concentration}Assume $(M,d,\m)$
is non-degenerate. Then any $c_{p}$-cyclically monotone set $\Gamma$
such that $\Gamma_{1}=p_{2}(\Gamma)$ is finite is non-degenerate.
In particular, if $\pi$ is a $p$-optimal coupling with $(p_{1})_{*}\pi\ll\m$
and $(p_{2})_{*}\pi=\sum_{i=1}^{n}a_{i}\delta_{x_{i}}$ then any measurable
set $\Gamma$ of full $\pi$-measures contains a non-degenerate, $c_{p}$-cyclically
monotone Borel set $\tilde{\Gamma}\subset\supp\pi\cap\Gamma$ of full
$\pi$-measure.
\end{lem}
\begin{proof}
Observe 
\[
\Gamma=\bigcup_{i=1}^{n}B^{i}\times\{x_{i}\}
\]
for measurable subsets $B^{i}\subset M$. Since $\Gamma_{1}=\m(\cup_{i=1}^{n}B^{i})$
the condition $\m(\Gamma_{0}^{A})>0$ implies there is at least one
$i\in\{1,\ldots,n\}$ with $\m(B^{i}\cap A)>0$. But then
\[
\m(\Gamma_{t}^{A})\ge\m(B_{t,x_{i}}^{i})>0
\]
showing that $\Gamma$ is non-degenerate.

For the last statement note that if $\pi(\Gamma')=1$ then $\pi(\Gamma'\cap\supp\pi)=1$
so that there is a Borel set $\tilde{\Gamma}\subset\Gamma'\cap\supp\pi$
with $\pi(\tilde{\Gamma})=1$. Since 
\[
p_{2}(\tilde{\Gamma})\subset p_{2}(\supp\pi)=\{x_{i}\}_{i=1}^{n}
\]
 is finite, the result follows.
\end{proof}
The same argument also holds more general if the background measure
is strongly non-degenerate. 
\begin{lem}
\label{lem:non-deg-implies-non-deg-concentration}If $(M,d,\m)$ is
strongly non-degenerate $\sND_{p}$ then for any $p$-optimal coupling
$\pi$ with $(p_{1})_{*}\pi\ll\m$ the following holds: Whenever $\pi(\Gamma')=1$
for a $\pi$-measurable set $\Gamma'$ then there is a $c_{p}$-cyclically
monotone, non-degenerate Borel subset $\Gamma\subset\Gamma'$ of full
$\pi$-measure.
\end{lem}
\begin{proof}
Just note that $\pi$ is concentrated on $\Gamma'\cap\supp\pi$ which
is measurable, non-degenerate and $c_{p}$-cyclically monotone. Hence
there is a Borel subset $\Gamma\subset\Gamma'\cap\supp\pi$ with $\pi(\Gamma)=1$.
By strong non-degenericity $\Gamma$ is non-degenerate.
\end{proof}
The two lemmas allow us to apply Theorem \ref{thm:abs-cts-interpolation}
and Corollary \ref{cor:abs-cts-interpolation}.
\begin{cor}
\label{cor:abs-cts-interpolation-under-non-deg}Let $(M,d,\m)$ be
a proper metric measure space and $\mu_{0},\mu_{1}\in\mathcal{P}_{p}(M)$
with $\mu_{0}\ll\m$. Assume either $(M,d,\m)$ is non-degenerate
and $\mu_{1}=\sum_{i=1}^{n}\lambda_{i}\delta_{x_{i}}$ or that $(M,d,\m)$
is strongly non-degenerate $\sND_{p}$. Then for any $0<t<s<1$ there
is a $p$-optimal dynamical coupling $\sigma$ between $\mu_{0}$
and $\mu_{1}$ with $(e_{t})_{*}\sigma,(e_{s})_{*}\sigma\ll\m$ and
$(e_{t})_{*}\sigma$ is maximal in the sense of Theorem \ref{thm:abs-cts-interpolation}.
In particular, $(M,d,\m)$ is strongly non-degenerate $\sND_{p}$
if and only if it has the interpolation property $\IP_{p}$.
\end{cor}
\begin{proof}
For strongly non-degenerate measures the result follows immediately
from the previous lemma. 

For the case of $\m$ being non-degenerate and $\mu_{1}=\sum_{i=1}^{n}\lambda_{i}\delta_{x_{i}}$
just observe that for any $p$-optimal coupling $\pi_{t,1}$ between
$\mu_{t}\ll\m$ and $\mu_{1}$ is strongly consistent. Indeed, if
$\tilde{\pi}\ll\pi_{t,1}$ then the set $p_{2}(\supp\tilde{\pi}\cap\Gamma')$
is finite hence contains a $c_{p}$-cylically monontone Borel set
$\Gamma$ with $\tilde{\pi}(\Gamma)=1$. 
\end{proof}
\begin{proof}
[Proof of Lemma \ref{lem:cpt-t-midpoints}]Since $(M,d)$ is geodesic
$\mathcal{B}$ is non-empty and closed in $\mathcal{P}_{p}(M).$ Furthermore,
properness of $M$ together with $\mathcal{B}$ being bounded implies
that $\mathcal{B}$ is weakly precompact. If $\mu_{t}^{0}$ and $\mu_{t}^{1}$
are measures in $\mathcal{B}$ then there are two $p$-optimal dynamical
couplings $\sigma^{0}$ and $\sigma^{1}$ such that $(e_{t})_{*}\sigma^{i}=\mu_{t}^{i}$
and $(e_{0},e_{1})\sigma^{i}$, $i=0,1$. 

It holds
\begin{align*}
\mu_{t}^{\lambda} & =(e_{t})_{*}\sigma^{\lambda}
\end{align*}
for 
\begin{align*}
\mu_{t}^{\lambda} & =(1-\lambda)\mu_{t}^{0}+\lambda\mu_{t}^{1}\\
\sigma^{\lambda} & =(1-\lambda)\sigma^{0}+\lambda\sigma^{1}
\end{align*}
so that 
\begin{align*}
W_{p}^{p}(\mu_{t}^{\lambda},\mu_{1}) & \le\int d(\gamma_{t},\gamma_{1})^{p}d\sigma^{\lambda}(\gamma)\\
 & =(1-\lambda)\int t^{p}d(\gamma_{0},\gamma_{1})^{p}d\sigma^{0}(\gamma)+\lambda\int t^{p}d(\gamma_{0},\gamma_{1})^{p}d\sigma^{1}(\gamma)\\
 & =t^{p}W_{p}^{p}(\mu_{0},\mu_{1})
\end{align*}
and similarly $W_{p}^{p}(\mu_{0},\mu_{t}^{\lambda})\le(1-t)^{p}W_{p}^{p}(\mu_{0},\mu_{1})$.
Thus $W_{p}(\mu_{0},\mu_{t}^{\lambda})+W_{p}(\mu_{t}^{\lambda},\mu_{1})\le W_{p}(\mu_{0},\mu_{1})$
which implies that $\mu_{t}^{\lambda}$ is a $t$-midpoint. Since
$(e_{0},e_{1})_{*}\sigma^{\lambda}=\pi$ we have $\mu_{t}^{\lambda}\in\mathcal{B}$
implying linear convexity of $\mathcal{B}$.

Now let $(\mu_{t}^{n})_{n\in\mathbb{N}}$ be a sequence in $\mathcal{B}$.
By weak compactness we can assuming after picking a subsequence and
relabeling that $(\mu_{t}^{n})_{n\in\mathbb{N}}$ converges weakly
to some $\mu\in\mathcal{P}(M)$. Since 
\begin{align*}
W_{p}(\mu_{0},\mu) & \le\liminf_{n\to\infty}W_{p}(\mu_{0},\mu_{t}^{n})=tW_{p}(\mu_{0},\mu_{1})\\
W_{p}(\mu,\mu_{1}) & \le\liminf_{n\to\infty}W_{p}(\mu_{t}^{n},\mu_{1})=tW_{p}(\mu_{0},\mu_{1})
\end{align*}
and 
\[
W_{p}(\mu_{0},\mu_{1})\le W_{p}(\mu_{0},\mu)+W_{p}(\mu,\mu_{1})\le W_{p}(\mu_{0},\mu_{1})
\]
we see that $\mu$ is a $t$-midpoint as well. Hence $W_{p}(\mu_{0},\mu_{t}^{n})\to W_{p}(\mu_{0},\mu)$
so that the sequence $(\mu_{t}^{n})_{n\in\mathbb{N}}$ also converges
in the $p$-th moment. This shows that $\mu_{t}^{n}\to\mu$ in $\mathcal{P}_{p}(M)$
(see \cite[Definition 6.8]{Villani2008OptTrans}). Thus any sequence
in $\mathcal{B}$ has a subsequence converging in $\mathcal{P}_{p}(M)$.
In particular, $\mathcal{B}$ is compact in $\mathcal{P}_{p}(M)$.
\end{proof}

\section{Existence of transport maps}

In this section we want to prove the existence of transport maps using
a combined approach of \cite{CH2015NBTrans} and \cite{CM2016TransMapsMCP}. 

\subsection*{Qualitatively non-degenerate measures}

Non-degenericity and the GKS-\-Con\-struc\-tion in the previous
section imply that there are absolutely continuous interpolations
between $\mu_{0}\ll\m$ and $\mu_{1}=\sum\lambda_{i}\delta_{x_{i}}$.
However, the non-degenericity condition is too weak to use approximation
arguments for general $\mu_{1}$. For this we need the following uniform
variant which was introduced by Cavalletti\textendash Huesmann \cite{CH2015NBTrans}
and represents a weak form of the measure contraction condition $\MCP(K,N)$,
see e.g. \cite{Sturm2006a,CM2016TransMapsMCP} and references therein.
\begin{defn}
\label{def:qual-non-deg}The measure $\m$ is said to be \emph{qualitatively
non-degenerate} if for all $R>0$ and $x_{0}\in M$ there is a function
$f_{R,x_{0}}:(0,1)\to(0,\infty)$ with 
\[
\limsup_{t\to0}f_{R,x_{0}}(t)>\frac{1}{2}
\]
 such that for every measurable $A\subset B_{R}(x_{0})$ and all $x\in B_{R}(x_{0})$
and $t\in(0,1)$ it holds 
\[
\m(A_{t,x})\ge f_{R,x_{0}}(t)\m(A).
\]
\end{defn}
\begin{cor}
Any qualitatively non-degenerate measure is non-degenerate.
\end{cor}
The following proposition shows that qualitatvely non-degenerate spaces
are proper and make it possible to use GKS-Construction of the previous
section.
\begin{prop}
\label{prop:qual-non-deg-implies-doubling}A qualitatively non-degenerate
measure $\m$ is locally doubling, i.e. for each $R>0$ and $x_{0}\in M$
there is a constant $C_{R,x_{0}}>0$ such that 
\[
\m(B_{2r}(x))\le C_{R,x_{0}}\cdot\m(B_{r}(x))
\]
whenever $B_{2r}(x)\subset B_{R}(x_{0})$. In particular, $(M,d)$
is a proper metric space.
\end{prop}
\begin{proof}
Just note that $B_{r}(x)\subset(B_{2r}(x))_{\frac{1}{2},x}$ for all
$x\in M$ and $r>0$. Thus qualitative non-degenericity implies for
$B_{2r}(x)\subset B_{R}(x_{0})$ 
\[
\m(B_{2r}(x))\le\frac{1}{f_{R,x_{0}}(\frac{1}{2})}\m(B_{r}).
\]
Finally, properness follows from $\m$ being locally doubling (see
e.g. \cite{Heinonen2001}).
\end{proof}
\begin{lem}
\label{lem:no-overlap-non-deg}Assume $(M,d,\m)$ is $p$-essentially
non-branching for some $p\in(1,\infty)$ and $\m$ is qualitatively
non-degenerate. If for a Borel set $A$, the set $A\times\{x,y\}$
is $c_{p'}$-cyclically monotone for $x\ne y\in M$ and $p'\in(1,\infty)$
then $\m(A)=0$.
\end{lem}
\begin{rem*}
One may replace the qualitative non-degenericity by the following
pointwise variant
\[
\liminf_{t\to0}\m(A_{t,x})>\frac{1}{2}\m(A).
\]
This condition is, however, too weak to do approximations of general
$c_{p}$-cyclically monotone sets as in Lemma \ref{lem:p-ess-nb-implies-p-str-non-deg}. 
\end{rem*}
\begin{proof}
By inner regularity we can assume $A$ is compact and $\{x,y\}\cup A\subset B_{R}(x_{0})$
for some $R>0$. By compactness of $A$ we find $t$ close to $0$
and $s$ close to $1$ such that $f_{R,x_{0}}(t)\ge\frac{1}{2}+\epsilon$,
\begin{align*}
\m(A_{\delta}) & \le(1+\epsilon)\m(A)\\
A_{t,x}\cup A_{t,y} & \subset A_{\delta}
\end{align*}
 and 
\[
(A_{s,x})_{\epsilon}\cap(A_{s,y})_{\epsilon}=\varnothing
\]
for some $\epsilon,\delta>0$. 

Decompose $A$ into two Borel sets 
\[
A^{eq}=\{z\in A\thinspace|\thinspace d(z,x)=d(z,y)\}
\]
and 
\[
A^{ne}=A\backslash A^{eq}=\{z\in A\thinspace|\thinspace d(z,x)\ne d(z,y)\}.
\]
It suffices to show that the claim is true for the cases $A=A^{eq}$
and $A=A^{ne}$. 

First assume $A=A^{ne}$ and observe that by $c_{p'}$-cyclic monotonicity
and the fact that $d(z,x)\ne d(z,y)$ for all $z\in A$ it holds 
\[
A_{t,y}\cap A_{t,x}=\varnothing\quad\text{for all }t\in(0,1).
\]
Hence 
\begin{align*}
(1+\epsilon)\m(A) & >\m(A_{\delta})\\
 & \ge\m(A_{t,x}\cup A_{t,y})\\
 & =\m(A_{t,x})+\m(A_{t,y})\\
 & \ge2f_{R,x_{0}}(t)\m(A)=(1+2\epsilon)\m(A)
\end{align*}
which implies that $\m(A)=0$.

For the case $A=A^{eq}$ assume by contradiction $\m(A)>0$. Set

\[
\mu_{0}=\frac{1}{\m(A)}\m\big|_{A}
\]
and observe that $A=A^{eq}$ implies that $A\times\{x,y\}$ is $c_{p"}$-cyclically
monotone for all $p"\in[1,\infty)$. In particular, $A\times\{x,y\}$
is $c_{p}$-cyclically monotone.

Apply Corollary \ref{cor:abs-cts-interpolation-under-non-deg} to
$(\mu_{0},\delta_{x})$ and $(\mu_{0},\delta_{y})$ to get two dynamical
couplings $\sigma^{x}$ and $\sigma^{y}$ whose interpolations at
times $s$ and $t$ are absolutely continuous. The choice of $s$
shows that $\mu_{s}^{x}=(e_{s})_{*}\sigma^{x}$ and $\mu_{s}^{y}=(e_{s})_{*}\sigma^{y}$
have disjoint support. Furthermore, the measures $\mu_{t}^{x}=(e_{t})_{*}\sigma^{x}$
and $\mu_{t}^{y}=(e_{t})_{*}\sigma^{y}$ are maximal with respect
to $\tilde{A}\times\{x\}$ and resp. $\tilde{A}\times\{y\}$ for some
$\tilde{A}\subset A$ with $\m(A\backslash\tilde{A})=0$. Since the
set $\tilde{A}\times\{x,y\}$ is still $c_{p}$-cyclically monotone,
the dynamical coupling $\frac{1}{2}(\sigma^{x}+\sigma^{y})$ is $p$-optimal
between $\mu_{0}$ and $\frac{1}{2}(\delta_{x}+\delta_{y})$. Because
$\mu_{s}=\frac{1}{2}(\mu_{s}^{x}+\mu_{s}^{y})$ is a decomposition
into mutually singular measures, Theorem \ref{thm:ess-nb-summary}
shows 
\[
\mu_{t}^{x}\perp\mu_{t}^{y}.
\]
By maximality of $\mu_{t}^{x}$ and $\mu_{t}^{y}$ it holds $\m\big|_{\tilde{A}_{t,x}}\ll\mu_{t}^{x}$
and $\m\big|_{\tilde{A}_{t,y}}\ll\mu_{t}^{y}$ so that $\m(\tilde{A}_{t,x}\cap\tilde{A}_{t,y})=0$.
In particular, since $\m$ is qualitatively non-degenerate 
\begin{align*}
\m(\tilde{A}_{t,x}\cup\tilde{A}_{t,y}) & =\m(\tilde{A}_{t,x})+\m(\tilde{A}_{t,y})\\
 & \ge2f(t)\m(A)\ge\left(1+2\epsilon\right)\m(A).
\end{align*}
Combining those facts we obtain 
\begin{align*}
(1+\epsilon)\m(A) & >\m(A_{\delta})\\
 & \ge\m(\tilde{A}_{t,x}\cup\tilde{A}_{t,y})\\
 & \ge(1+2\epsilon)\m(A)
\end{align*}
which is a contradiction. This shows that $\m(A)=0$.
\end{proof}
\begin{cor}
Assume $(M,d,\m)$ is $p$-essentially non-branching and $\m$ qualitatively
non-degenerate. If for some $p'\in(1,\infty)$ the set $\Gamma$ is
a $c_{p'}$-cyclically monotone set in $B_{R}(x_{0})\times B_{R}(x_{0})$
and $\Gamma_{1}$ is finite then 
\[
\m(\Gamma_{t})\ge f_{R,x_{0}}(t)\m(\Gamma_{0}).
\]
\end{cor}
\begin{proof}
Let $\{x_{i}\}_{i=1}^{n}=\Gamma_{1}$ with $x_{i}\ne x_{j}$ whenever
$i\ne j$ and set $\Gamma^{i}=\Gamma\cap(M\times\{x_{i}\})$. Then
the previous theorem shows 
\[
\m(\Gamma_{0}^{i}\cap\Gamma_{0}^{j})=0
\]
for $i\ne j$. Thus there are disjoint sets $A^{i}\subset\Gamma^{i}$,
$i=1,\ldots,n$, such that 
\[
\m(\Gamma_{0}\backslash\cup A^{i})=0.
\]
Similarly one may replace $A^{i}$ by a possibly smaller set which
has full $\m$-measure in $A^{i}$ satisfying the condition above
and, in addition, it holds
\[
\m(A_{t,x_{i}}^{i}\cap A_{t,x_{j}}^{j})=0
\]
for $i\ne j$. Setting 
\[
\Gamma'=\bigcup_{i=1}^{n}\{x_{i}\}\times A^{i}
\]
we conclude 
\begin{align*}
\m(\Gamma_{t}) & \ge\m(\Gamma_{t}^{'})=\sum_{i=1}^{n}\m(A_{t,x_{i}}^{i})\\
 & \ge\sum_{i=1}^{n}f_{R,x_{0}}(t)\m(A^{i})=f_{R}(t)\m(\Gamma_{0}).
\end{align*}
\end{proof}
A similar argument also shows the following. As the result is not
used below we leave the proof to the interested reader.
\begin{cor}
\label{cor:maps-between-abs-to-discrete}Assume $(M,d)$ is $p$-essentially
non-branching and $\m$ is qualitatively non-degenerate. Then for
any $p'\in(1,\infty)$ and any $p'$-optimal coupling $\pi$ with
$(p_{1})_{*}\pi\ll\m$ and $(p_{2})_{*}\pi=\sum a_{i}\delta_{x_{i}}$
is induced by a transport map.
\end{cor}
Via an approximation argument of Cavalletti\textendash Huesmann \cite[Proposition 4.3]{CH2015NBTrans}
qualitative non-degenericity implies strong non-degneraticity.
\begin{lem}
\label{lem:p-ess-nb-implies-p-str-non-deg}Assume $(M,d,\m)$ is $p$-essentially
non-branching and $\m$ is qualitatively non-degenerate. Then for
any $c_{p'}$-cyclically monotone Borel set $\Gamma$ in $B_{R}(x_{0})\times B_{R}(x_{0})$
it holds 
\[
\m(\Gamma_{t})\ge f_{R,x_{0}}(t)\m(\Gamma_{0}).
\]
In particular, $(M,d,\m)$ is strongly non-degenerate $\sND_{p'}$
for all $p'\in(1,\infty)$.
\end{lem}
\begin{proof}
For compact $\Gamma$ the argument is a in \cite[Proposition 4.3]{CH2015NBTrans}.
For completeness, we present the argument: Let $(\Gamma^{(n)})_{n\in\mathbb{N}}$
be a sequence of $c_{p}$-cyclically monotone sets such that $\Gamma_{0}^{(n)}=\Gamma_{0}$
and $\Gamma_{1}^{(n)}\subset\Gamma_{1}$ is finite. More precisely,
choose a countable dense sequence $y_{n}\in\Gamma_{1}$ and define
\[
E_{i}^{(n)}=\{x\in\Gamma_{0}\,|\,d(x,y_{i})^{p}-\varphi^{c_{p}}(y_{i})\le d(x,y_{j})^{p}-\varphi^{c_{p}}(y_{j}),j=1,\ldots,n\}
\]
and 
\[
\Gamma^{(n)}=\bigcup_{i=1}^{n}E_{i}^{(n)}\times\{y_{i}\}.
\]
From the definition of $E_{i}^{(n)}$ it follows that $\Gamma^{(n)}$
is $c_{p'}$-cyclically monotone. Furthermore, compactness of $\Gamma$
shows that for all $\epsilon>0$ there is an $N_{\epsilon}$ such
that for all $n\ge N_{\epsilon}$ it holds 
\[
\Gamma_{t}^{(n)}\subset(\Gamma_{t})_{\epsilon}=\bigcup_{x\in\Gamma_{t}}B_{\epsilon}(x)
\]
This yields immediately the result for compact $\Gamma$ as follows
\begin{align*}
\m(\Gamma_{t}) & =\lim_{\epsilon\to0}\m((\Gamma_{t})_{\epsilon})\\
 & \ge\limsup_{n\to\infty}\m(\Gamma_{t}^{(n)})\\
 & \ge f_{R,x_{0}}(t)\m(\Gamma_{0}).
\end{align*}

For arbitrary $c_{p}$-monotone Borel sets $\Gamma$ in $B_{R}(x_{0})\times B_{R}(x_{0})$
we can use the Measurable Selection Theorem and Lusin's Theorem to
reduce the result to compact set. If $\m(\Gamma_{0})=0$ then there
is nothing to prove. So assume $\m(\Gamma_{0})>0$. Choose a measurable
selection $T$ of $\Gamma$. As $\m$ is locally bounded and $\Gamma_{0}\subset B_{R}(x_{0})$
we obtain by Lusin's Theorem a family of compact sets $K^{1}\subset K^{2}\subset\ldots\subset\Gamma_{0}$
with $\m(K^{i})\to\m(\Gamma_{0})$ such that $T$ restricted to $K^{i}$
is continuous. Define $\Gamma^{i}=\operatorname{graph}_{K^{i}}T\subset\Gamma$
and note that $\Gamma^{i}$ is compact with $\Gamma_{0}^{i}=K^{i}$
so that 
\begin{align*}
\m(\Gamma_{t}) & \ge\limsup_{i\to\infty}\m(\Gamma_{t}^{i})\\
 & \ge\limsup_{i\to\infty}f_{R,x_{0}}(t)\m(K^{i})=f_{R,x_{0}}(t)\m(\Gamma_{0}).
\end{align*}

It remains to show that $\m$ is strongly non-degenerate $\sND_{p}$.
First observe
\[
\bigcup_{R>0}\Gamma^{R}=\Gamma
\]
where $\Gamma^{R}=\Gamma\cap(B_{R}(x_{0})\times B_{R}(x_{0}))$. As
$\Gamma^{R}$ is bounded we have 
\[
\m(\Gamma_{t})\ge\m(\Gamma_{t}^{R})\ge f_{R,x_{0}}(T)\m(\Gamma_{0}^{R}).
\]
Assume now $\m(\Gamma_{0})>0$ then $\m(\Gamma_{0}^{R})\in(0,\infty)$
for all large $R>0$ so that 
\[
\m(\Gamma_{t})\ge\m(\Gamma_{t}^{R})>0.
\]

Non-degenericity of $\Gamma$ follows by observing that for all Borel
sets $A$ the set $\Gamma^{A,0}=\Gamma\cap(A\times M)$ is still a
$c_{p}$-cyclically monotone Borel set. Thus $\m(\Gamma_{0}^{A,0})>0$
implies $\m(\Gamma_{t}^{A,0})>0$.
\end{proof}
\begin{rem*}
The proof of the results above relies only on the qualitative non-degenericity
of $\m$ and that 
\[
\m(\{z\in M\thinspace|\thinspace d(z,x)=d(z,y)\})=0
\]
for all $x\ne y$. In particular, it holds for $(\mathbb{R}^{n},\|\cdot-\cdot\|_{\infty},\lambda^{n})$
which is highly branching. 
\end{rem*}
By combining the previous lemma and the idea of the proof of Lemma
\ref{lem:no-overlap-non-deg} we obtain the main theorem of this section.
\begin{thm}
\label{thm:MCPimpliesGTB}Assume $(M,d,\m)$ is $p$-essentially non-branching
and $\m$ is qualitatively non-degenerate. Then any $p$-optimal coupling
$\pi\in\mathcal{P}_{p}(M\times M)$ with $(p_{1})_{*}\pi\ll\m$ is
induced by a transport map. In particular, any such space has good
transport behavior $\GTB_{p}$. 
\end{thm}
Combined with Proposition \ref{prop:GTB-ENB} and the existence of
absolutely continuous interpolations (Corollary \ref{cor:abs-cts-interpolation-under-non-deg})
we get the following two corollaries. 
\begin{cor}
If $(M,d,\m)$ is $p$-essentially non-branching and $\m$ qualitatively
non-degenerate then between any two measure $\mu_{0},\mu_{1}\in\mathcal{P}_{p}(M)$
with $\mu_{0}\ll\m$ there is a unique $p$-optimal dynamical coupling
$\sigma$ and this coupling satisfies $(e_{t})_{*}\sigma\ll\m$ for
all $t\in[0,1)$. In particular, it has the strong interpolation property
$\sIP_{p}$.
\end{cor}
\begin{cor}
Assume $\m$ is qualitatively non-degenerate. Then $\m$ is $p$-essentially
non-branching if any only if it has good transport behavior $\GTB_{p}$
.
\end{cor}
\begin{proof}
[Proof of the Theorem]Note by Proposition \ref{prop:qual-non-deg-implies-doubling},
$(M,d)$ is proper so that we can apply the GKS-Construction of the
previous section. 

Let $\Gamma=\supp\pi$ and note that $\m(p_{1}(\Gamma))=\m(\supp((p_{1})_{*}\pi))>0$.
It suffices to show that $\Gamma(x)$ is single-valued for $\m$-almost
all $x\in M$. This holds, if for all $R>0$, $\Gamma^{R}(x)$ is
single-valued for $\m$-almost all $x\in M$ where $\Gamma^{R}=\Gamma\cap(\bar{B}_{R}(x_{0})\times\bar{B}_{r}(x_{0}))$.
Note that large $R>0$ it holds $\m(p_{1}(\Gamma^{R}))\in(0,\infty)$. 

Assume, by contradiction, that for some $R>0$ there is a Borel set
$A$ with $\m(A)>0$ and the set $\Gamma^{R}(x)$ is non-empty and
not single-valued for all $x\in A$. Then by the Selection Dichotomy
of Sets (Theorem \ref{thm:selection-dichotomy-set}), there is a compact
set $K\subset A$ of positive $\m$-measure, and two continuous maps
$T_{1},T_{2}:M\to M$ with $T_{1}(K)\cap T_{2}(K)=\varnothing$ and
\[
(x,T_{1}(x)),(x,T_{2}(x))\in\Gamma^{R}\subset\supp\pi\cap(\bar{B}_{R}(x_{0})\times\bar{B}_{R}(x_{0})).
\]
Restricting $K$ further, we can also assume $\supp(\m\big|_{K})=K\subset\bar{B}_{R}(x_{0})$. 

Define now $\mu_{0}=\frac{1}{\m(K)}\m\big|_{K}$, $\pi_{i}=(\operatorname{id}\times T_{i})_{*}\mu_{0}$
and $\mu_{1}^{i}=(p_{2})_{*}\pi_{i}$ for $i=1,2$. Let $\Gamma^{(i)}=\supp\pi_{i}$,
$i=1,2$, and note that $\Gamma_{0}^{(i)}=K$ and both $\Gamma^{(1)}$
and $\Gamma^{(2)}$ are compact and $c_{p}$-cyclically monotone. 

Choose $\delta>0$, $t$ close to $0$ and $s$ close to $1$ such
that $f_{R,x_{0}}(t)\ge\frac{1}{2}+\epsilon$, 
\begin{align*}
\m(K_{\delta}) & \le(1+\epsilon)\m(K)\\
\Gamma_{t}^{(1)}\cup\Gamma_{t}^{(2)} & \subset\tilde{A}_{\delta}
\end{align*}
 and 
\[
(\Gamma_{s}^{(1)})_{\epsilon}\cap(\Gamma_{s}^{(2)})_{\epsilon}=\varnothing
\]

Corollary \ref{cor:abs-cts-interpolation-under-non-deg} applied to
$(\mu_{0},\mu_{1}^{1})$ and $(\mu_{0},\mu_{1}^{2})$ gives two $p$-optimal
dynamical couplings $\sigma^{(1)}$ and $\sigma^{(2)}$ such that
$(e_{s})_{*}\sigma^{(1)}$ and $(e_{s})_{*}\sigma^{(2)}$ are absolutely
continuous with respect to $\m$ and have disjoint support. The choice
of $\Gamma^{(1)}$ and $\Gamma^{(2)}$ implies that $\frac{1}{2}(\pi_{1}+\pi_{2})$
is supported on $\Gamma^{(1)}\cup\Gamma^{(2)}\subset\supp\pi$. Hence
$(\restr_{0,s})_{*}\frac{1}{2}(\sigma^{(1)}+\sigma^{(2)})$ is a $p$-optimal
dynamical coupling between $\mu_{0}$ and $\frac{1}{2}(\mu_{s}^{1}+\mu_{s}^{2})$
so that Theorem \ref{thm:ess-nb-summary} shows 
\[
\mu_{t}^{(1)}=(e_{t})_{*}\sigma^{(1)}\bot(e_{t})_{*}\sigma^{(2)}=\mu_{t}^{(2)}.
\]
Maximality at time $t$ shows that for $i=1,2$ there are measurable
subsets $\tilde{\Gamma}^{(i)}\subset\Gamma^{(i)}$ with $\m(K\backslash\tilde{\Gamma}_{0}^{(i)})=0$
and $\m\big|_{\tilde{\Gamma}_{t}^{(i)}}\ll\mu_{t}^{(i)}\ll\m\big|_{\tilde{\Gamma}_{t}^{(i)}}$.
Since $\mu_{t}^{(1)}\bot\mu_{t}^{(2)}$ we must have $\m(\tilde{\Gamma}_{t}^{(1}\cap\tilde{\Gamma}_{t}^{(2)})=0$.
In combination with Lemma \ref{lem:p-ess-nb-implies-p-str-non-deg}
this yields
\[
\m(\tilde{\Gamma}_{t}^{(1)}\cup\tilde{\Gamma}_{t}^{(2)})=\m(\tilde{\Gamma}_{t}^{(1)})+\m(\tilde{\Gamma}_{t}^{(2)})\ge2f_{R}(t)\m(K).
\]
This, however, leads to the following contradiction
\begin{align*}
(1+\epsilon)\m(K) & >\m(K_{\delta})\\
 & \ge\m(\tilde{\Gamma}_{t}^{(1)}\cup\tilde{\Gamma}_{t}^{(2)})\\
 & =\m(\tilde{\Gamma}_{t}^{(1)})+\m(\tilde{\Gamma}_{t}^{(2)})\\
 & \ge2f_{R,x_{0}}(t)\m(K)=(1+2\epsilon)\m(K).
\end{align*}
Thus we have proved that the $\Gamma^{R}(x)$ at most single-valued
for $\m$-almost all $x\in M$ proving that $\pi$ is induced by a
transport map.
\end{proof}
The proof relies heavily on the $p$-essentially non-branching property
of dynamical couplings between absolutely continuous measures. In
contrast to the case of a discrete target measures we cannot show
that that general $p'$-optimal couplings with absolutely continuous
first marginals are induced by transport maps. Nevertheless, $p$-essentially
non-branching and the idea of Lemma \ref{lem:no-overlap-non-deg}
still exclude a too general behavior of the support of $p'$-optimal
couplings.
\begin{thm}
\label{thm:initial-double-implies-ae-equal-dist}Assume $(M,d,\m)$
is $p$-essentially non-branching for some $p\in(1,\infty)$, $\m$
is qualitatively non-degenerate and $p'\in(1,\infty)$. Then for any
$p'$-optimal $\pi\in\mathcal{P}(M\times M)$ with $(p_{1})_{*}\pi\ll\m$
and for $\mu_{0}$-almost every $x\in M$ it holds 
\[
d(x,y_{1})=d(x,y_{2})\quad\text{whenever }(x,y_{1}),(x,y_{2})\in\supp\pi.
\]
\end{thm}
\begin{cor}
The $c_{p'}$-superdifferential $\partial^{c_{p'}}\varphi$ of a $c_{p'}$-concave
function $\varphi$ satisfies for $\m$-almost every $x\in M$
\[
d(x,y_{1})=d(x,y_{2})\quad\text{for all }y_{1},y_{2}\in\partial^{c_{p'}}\varphi(x).
\]
\end{cor}
\begin{rem*}
The property $p$-essentially non-branching is used only to show that
$c_{p'}$-cyclically monotone sets are non-degenerate. As mentioned
above, this holds if we replace $p$-essentially non-branching by
the assumption 
\[
\m(\{z\in M\thinspace|\thinspace d(z,x)=d(z,y)\})=0
\]
for all $y\ne z$,
\end{rem*}
\begin{proof}
If the claim was false then $\pi$ is not induced by a transport map
and as above we get a compact set $K$ of positive $\m$-measure and
measurable selections $T_{1}$ and $T_{2}$ as above which, in addition,
satisfy 
\[
\sup_{(x,y_{1})\in K\times T_{1}(K)}d(x,y_{1})<\inf_{(x,y_{2})\in K\times T_{2}(K)}d(x,y_{2}).
\]
Let $\mu_{0}=\frac{1}{\m(K)}\m\big|_{K}$ and for $i=1,2$ define
$\pi_{i}=(\operatorname{id}\times T_{i})_{*}\tilde{\mu}_{0}$ and
$\Gamma^{(i)}=\supp\pi_{i}$. Again $\Gamma^{(i)}$ is $c_{p}$-cyclically
monotone, but satisfies, in addition, the following
\[
\Gamma_{t}^{(1)}\cap\Gamma_{t}^{(1)}=\varnothing\quad\text{for all }t\in(0,1)
\]
Choosing $\epsilon$, $\delta$ and $t$ as in the previous proof,
we arrive at the following contradiction
\begin{align*}
(1+\epsilon)\m(K) & >\m(K_{\delta})\\
 & \ge\m(\Gamma_{t}^{(1)}\cup\Gamma_{t}^{(2)})\\
 & =\m(\Gamma_{t}^{(1)})+\m(\Gamma_{t}^{(2)})\\
 & \ge2f_{R,x_{0}}(t)\m(K)=(1+2\epsilon)\m(K).
\end{align*}
\end{proof}

\subsection*{Density bounds of qualitatively non-degenerate measures}

In \cite{CM2016TransMapsMCP} Cavalletti\textendash Mondino showed
that the measure contraction property $\MCP(K,N)$ implies the existence
of an absolutely continuous interpolations with controlled $L^{\infty}$-bounds
on their density. This was then used to prove the general existence
of transport maps.
\begin{defn}
The measure $\m$ has \emph{bounded density property} if for all $R>0$
and $x_{0}\in M$ there is a function $g_{R,x_{0}}:(0,1)\to(0,\infty]$
with 
\[
\limsup_{t\to0}g_{R,x_{0}}(t)<2
\]
such that for some $p\in(1,\infty)$ and for every $\mu_{0}=f_{0}\m\in\mathcal{P}_{p}(M)$
with $\|f_{0}\|_{\infty}<\infty$, $\supp\mu_{0}\subset B_{R}(x_{0})$
and $x\in B_{R}(x_{0})$ there is a geodesic $t\mapsto\mu_{t}=f_{t}\m$
between $\mu_{0}$ and $\delta_{x}$ in $\mathcal{P}_{p}(M)$ such
that 
\[
\|f_{t}\|_{\infty}\le g_{R,x_{0}}(t)\|f_{0}\|_{\infty}.
\]
\end{defn}
\begin{rem*}
It is easy to see that the definition does not depend on $p\in(1,\infty)$.
\end{rem*}
First observe that the bounded density property is stronger than qualitative
non-degenericity.
\begin{lem}
\label{lem:bdd-dens-imply-qual-non-deg}Every measure $\m$ with bounded
density property is qualitatively non-degenerate. 
\end{lem}
\begin{proof}
Let $\mu_{0}=\frac{1}{\m(A_{0})}\m\big|_{A_{0}}$ and note that 
\[
\supp\mu_{t}\subset A_{t,x}
\]
and 
\[
\|f_{0}\|_{\infty}=\frac{1}{\m(A_{0})}
\]
we obtain 
\[
1=\int_{A_{t,x}}f_{t}\m\le g_{R,x_{0}}(t)\frac{1}{\m(A_{0})}\m(A_{t,x}).
\]
Choosing $f_{R,x_{0}}=g_{R,x_{0}}^{-1}$ we obtain the result. 
\end{proof}
The bounded density property was proven to hold for spaces with curvature
dimension condition $\CD(K,N)$ by Rajala \cite[Theorem 4.2]{Rajala2012a}
and later for the measure contraction property $\MCP(K,N)$ by Cavalletti\textendash Mondino
\cite[Theorem 3.1]{CM2016TransMapsMCP}. Assuming $\m$ is $p$-essentially
non-branching, the following result implies that the bounded density
property is equivalent to qualitative non-degenericity.
\begin{prop}
\label{prop:qND-ess-nb-implies-bdd-dens}Assume $(M,d,\m)$ be $p$-essentially
non-branching and $\m$ is qualitatively non-degenerate and $\mu_{0},\mu_{1}\in\mathcal{P}_{p}(M)$
with $\mu_{0}=f_{0}\m$ and $\supp\mu_{0},\supp\mu_{1}\subset B_{R}(x_{0})$.
Then for the unique $p$-optimal dynamical coupling $\sigma\in\OptGeo_{p}(\mu_{0},\mu_{1})$
it holds 
\[
f_{t}(\gamma_{t})\le\frac{1}{f_{R,x_{0}}(t)}f_{0}(\gamma_{0})\quad\text{for \ensuremath{\sigma}-almost all }\gamma\in\Geo(M,d)
\]
where $(e_{t})_{*}\sigma=f_{t}\m$. In particular, it holds
\[
\|f_{t}\|_{\infty}\le\frac{1}{f_{R,x_{0}}(t)}\|f_{0}\|_{\infty}
\]
so that $\m$ has the bounded density property. 
\end{prop}
\begin{cor}
In a $p$-essentially non-branching metric measure space $(M,d,\m)$
the following are equivalent:
\begin{itemize}
\item The measure $\m$ is qualitatively non-degenerate.
\item The measure $\m$ has the bounded density property.
\end{itemize}
\end{cor}
\begin{proof}
[Proof of the proposition]We first assume $f_{0}\equiv\frac{1}{\m(A_{0})}$.
If the claim was wrong then there is a compact set $\mathsf{L}\subset\Geo(M,d)$
with $\sigma(\mathsf{L})>0$ such that 
\[
\frac{1}{f_{R,x_{0}}(t)}\m(A_{0})\le(1-\epsilon)f_{t}(\gamma_{t})\qquad\text{for all }\gamma\in\mathsf{L}.
\]
In particular, by restricting $\sigma$ to $\mathsf{L}$ we see that
$\tilde{A}_{0}=e_{0}(\mathsf{L})\subset A_{0}$ it holds and 
\[
\frac{1}{f_{R,x_{0}}(t)}\m(\tilde{A}_{0})\le(1-\epsilon)\tilde{f}_{t}(\gamma_{t})\quad\text{for \ensuremath{\sigma}-almost all \ensuremath{\gamma\in\mathsf{L}}}
\]
where $\tilde{f}_{t}\m=(e_{t})_{*}\sigma_{\mathsf{L}}$. The qualitative
non-degenericity yields
\[
\m(e_{t}(\mathsf{L}))\ge f_{R,x_{0}}(t)\m(\tilde{A}_{0}).
\]
Note that we always have 
\[
\operatorname{ess\,inf}_{\m|e_{t}(\mathsf{L})}\tilde{f}_{t}\le\frac{1}{\m(e_{t}(\mathsf{L}))}.
\]
This, however, leads to the following contradiction 
\[
\operatorname{ess\,inf}_{\m|e_{t}(\mathsf{L})}f_{t}\le\frac{1}{f_{R,x_{0}}(t)}\m(\tilde{A}_{0})\le(1-\epsilon)\tilde{f}_{t}(\gamma_{t})\quad\text{for \ensuremath{\sigma}-almost all \ensuremath{\gamma\in\mathsf{L}}}.
\]

For general $\mu_{0}$, observe that 
\[
t\mapsto\tilde{\mu}_{t}=\frac{1}{\m(\{f_{0}>0\})}\int_{\{f_{0}>0\}}\frac{1}{f_{0}(x)}\delta_{T_{t}(x)}d\mu_{0}(x)
\]
is a geodesic in $\mathcal{P}_{p}(M)$ such that $\tilde{\mu}_{0}$
has constant density, i.e. $\tilde{\mu}_{0}=\frac{1}{\m(\{f_{0}>0\})}\m\big|_{\{f_{0}>0\}}$.
Furthermore, $\tilde{f}_{t}$ satisfies 
\[
\tilde{f}_{t}(\gamma_{t})=\frac{f_{t}(\gamma_{t})}{f_{0}(\gamma_{0})}\tilde{f}_{0}(\gamma_{0})
\]
so that 
\[
\frac{f_{t}(\gamma_{t})}{f_{0}(\gamma_{0})}\tilde{f}_{0}(\gamma_{0})=\tilde{f}_{t}(\gamma_{t})\le\frac{1}{f_{R,x_{0}}(t)}\tilde{f}_{0}(\gamma_{0})
\]
which proves the claim.
\end{proof}
Recall that the $\MCP(0,N)$-condition holds if for all $\mu_{0}=\rho_{0}\m\in\mathcal{P}_{2}(M)$
and all $x\in M$ there is geodesic $t\mapsto\mu_{t}=\rho_{t}\m+\mu_{t}^{s}$
between $\mu_{0}$ and $\delta_{x}$ such that 
\[
\int\rho_{t}^{1-\frac{1}{N}}\ge(1-t)\int\rho_{0}^{1-\frac{1}{N}}d\m.
\]
Cavalletti\textendash Mondino showed that $\MCP(0,N)$-spaces have
the bounded density property with $g_{R,x_{0}}(t)=(1-t)^{-N}$, see
\cite[Theorem 3.1]{CM2016TransMapsMCP}. Thus we obtain the following
equivalent characterization of essentially non-branching $\MCP(0,N)$-spaces.
\begin{cor}
A $p$-essentially non-branching metric measure space satisfies the
measure contraction property $\MCP(0,N)$ if and only if it is qualitatively
non-degenerate with $f_{R,x_{0}}(t)=(1-t)^{N}$, i.e. $\m(A_{t,x})\ge(1-t)^{N}\m(A)$
for all $x\in M$ and all Borel set $A\subset M$ of finite $\m$-measure.
\end{cor}
There are similar versions for the general measure contraction property
$\MCP(K,N)$, $K\in\mathbb{R}$ and $N\in[1,\infty)$. This actually
shows that one can regard the measure contraction property as a directional
version of Bishop\textendash Gromov volume comparison condition which
for $K=0$ says that $\m(B_{r}(x))\ge(1-t)^{N}\m(B_{(1-t)r}(x))$.

\begin{rem*}
[Removing essentially non-branching I] Using a construction of good
geodesics as in Rajala \cite{Rajala2012a} and Cavalletti\textendash Mondino
\cite{CM2016TransMapsMCP} combined with the GKS-Construction (Theorem
\ref{thm:abs-cts-interpolation}) it might be possible to obtain the
equivalence without assuming that measure is essentially non-branching.We
leave the details to a future work.
\end{rem*}
Along the lines of \cite{Rajala2012} we also obtain local versions
of the Poincaré inequality with constant 
\[
C_{R,x_{0}}=\sup_{t\in(0,1)}\min\{\frac{1}{f_{R,x_{0}}(t)},\frac{1}{f_{R,x_{0}}(1-t)}\},
\]
i.e. for all Lipschitz functions $f:M\to\mathbb{R}$ and $B_{r}(x)\subset B_{R}(x_{0})$
it holds
\[
\int_{B_{r}(x)}|f-\bar{f}_{B_{r}(x)}|d\m\le4rC_{R,x_{0}}\int_{B_{2r}(x)}\operatorname{Lip}fd\m
\]
where 
\[
\bar{f}_{A}=\frac{1}{\m(A)}\int_{A}fd\m
\]
and
\[
\operatorname{Lip}f(x)=\limsup_{y\to x}\frac{|f(y)-f(x)|}{d(x,y)}.
\]

\begin{rem*}
[Removing essentially non-branching II]Similar to Lemma \ref{lem:p-ess-nb-implies-p-str-non-deg}
it is possible to show that under the condition
\[
\m(\{z\in M\thinspace|\thinspace d(z,x)=d(z,y)\})=0
\]
for all $x\ne y$ the bounded density property holds between every
$\mu_{0},\mu_{1}\in\mathcal{P}_{p}(M)$ whenever $\supp\mu_{0},\supp\mu_{1}\subset B_{R}(x_{0})$
and the function $g_{R,x_{0}}$ is upper semi-continuous in $(0,1)$.
In particular, if the density bounds are sufficiently nice then a
local doubling condition and local Poincaré inequality holds. 

We quickly sketch the argment: Note first that Lemma \ref{lem:no-overlap-non-deg}
holds for those spaces so that one obtains for $\mu_{1}=\sum_{i=1}^{n}\lambda_{i}\delta_{x_{i}}$
a geodesic in $\mathcal{P}_{p}(M)$ between $\mu_{0}$ and $\mu_{1}$
which has uniform density bounds only depending on the density of
$\mu_{0}$. Now let $\mu_{1}^{n}\rightharpoonup\mu_{1}$. At a fixed
time $t\in(0,1)$ there is a $\mu_{t}^{n}=\rho_{t}^{n}\m$ with $\int\rho_{t}^{n}\m=1$
and $\|\rho_{t}^{n}\|_{\infty}\le g_{R,x_{0}}(t)\|\rho_{0}\|_{\infty}$
implying that $(\rho_{t}^{n})_{n\in\mathbb{N}}$ is precompact in
$L^{1}(\m)$. Hence up to extracting a subsequence $\rho_{t}^{n}\to\rho_{t}$
in $L^{1}(\m)$, $\|\rho_{t}\|_{\infty}\le C_{t}\|\rho_{0}\|_{\infty}$
and $\mu_{t}=\rho_{t}\m$ being a $t$-midpoint of $\mu_{0}$ and
$\mu_{1}$. The same argument then gives a geodesic $t\mapsto\mu_{t}$
which is absolutely continuous with uniform density bounds at all
points $t\in\mathbb{Q}\cap(0,1)$. By upper semi-continuity of $g_{R,x_{0}}$
and the again same argument this time applied to $\mu_{t_{n}}\rightharpoonup\mu_{t}$
with $t_{n}\in\mathbb{Q}\cap(0,1)$ and $t_{n}\to t\in(0,1)$ shows
that $\mu_{t}$ is absolute continuous with uniform density bound,
compare also with \cite[Proof of Theorem 4.1]{CM2016TransMapsMCP}.
\end{rem*}

\subsection*{Generalizations to $N=\infty$}

As it turns out the idea of the proof of existence of transport maps
can be easily generalized to a more general situation. Compare the
results of this section with \cite[Theorem 3.3(ii)]{Gigli2012} where
non-branching spaces were treated. Recall that the $\CD_{p}(K,\infty)$-condition
(see \cite{LottVillani2009,Sturm2006,Kell2015}) requires that for
$\mu_{0},\mu_{1}\in\mathcal{P}_{p}(M)$ with $\mu_{i}\ll\m$ there
is a geodesic $t\mapsto\mu_{t}\ll\m$ such that 
\[
\int f_{t}\log f_{t}d\m\le(1-t)\int f_{0}\log f_{0}d\m+t\int f_{1}\log f_{1}d\m-Kt(1-t)W_{p}(\mu_{0},\mu_{1})^{2}
\]
where $f_{t}$ is the density of $\mu_{t}$. 

If we choose $\mu_{0}=\frac{1}{\m(A_{0})}\m\big|_{A_{0}}$ and apply
Jensen's inequality on the left-hand side, then it holds
\[
\log\m(\Gamma_{t})\ge(1-t)\log\m(A_{0})-t\int f_{1}\log f_{1}d\m+Kt(1-t)W_{2}(\mu_{0},\mu_{1})^{2}
\]
where $\Gamma_{t}=\operatorname{supp}\mu_{t}$. Thus 
\[
\lim_{t\to0}\m(\Gamma_{t})=\m(A_{0})
\]
whenever $\int f_{1}\log f_{1}d\m<\infty$ and $A_{0}$ is compact.
\begin{rem*}
We can replace the $\CD_{p}(K,\infty)$-condition by the $\CD_{p}^{*}(K,N)$-condition
with $N<0$ as defined by Ohta in \cite{Ohta2016}. Indeed, following
the proof of \cite[Theorem 4.8]{Ohta2016} gives a stronger variant
of the Brunn\textendash Minkowski inequality (replace $A_{t}$ by
$\Gamma_{t}$) which for $K=0$ and $r=-\frac{1}{N}>0$ says 
\[
\m(\Gamma_{t})^{-r}\le(1-t)\m(A_{0})^{-r}-t\int f_{1}^{1+r}d\m.
\]
implying again $\lim_{t\to0}\m(\Gamma_{t})=\m(A_{0})$.
\end{rem*}
\begin{lem}
Let $A$ be a bounded Borel set and $\mu_{1}\in\mathcal{P}_{p}(M)$
with $\mu_{1}\ll\m$. Assume $(M,d,\m)$ is $p$-essentially non-branching
and satisfies the $\CD_{p}(K,\infty)$-condition. If the geodesic
connecting $\mu_{0}=\frac{1}{\m(A)}\m\big|_{A}$ and $\mu_{1}$ is
unique then the (unique) $p$-optimal coupling $\pi$ of $\mu_{0}$
and $\mu_{1}$ is induced by a transport map.
\end{lem}
\begin{rem*}
Strictly speaking the assumptions imply that the strong $\CD_{p}(K,\infty)$-condition
holds between $\mu_{0}$ and $\mu_{1}$ thus the argument of the proof
of \cite[Corollary 1.4]{RS2014NonbranchStrongCD} can be used. For
completeness we present the arguments based on the ideas above.
\end{rem*}
\begin{proof}
Assume by contradiction that the claim is false for $\mu_{0}=\frac{1}{\m(A)}\m\big|_{A}$
and $\mu_{1}\ll\m$. Then by the Selection Dichotomy (Theorem \ref{thm:selection-dichotomy-product-measures})
there are a compact set $K\subset A$ and two disjoint bounded closed
set $A_{1}$ and $A_{2}$ such that $\pi(K\times A_{1})>0$ and 
\begin{align*}
\mu_{0} & =(p_{1})_{*}\pi_{i}=\frac{1}{\m(K)}\m\big|_{K}.\\
\mu_{1}^{i} & =(p_{2})_{*}\pi_{i}\ll\mu_{1}\ll m
\end{align*}
where $\pi_{i}=\frac{1}{\pi(K\times A_{i})}\pi\big|_{K\times A_{i}}$
and $i=1,2$. Denote the density of $\mu_{1}^{i}$ by $f_{1}^{i}$
and note that for large $n\in\mathbb{N}$
\[
\mu_{1}^{1}(\{f_{1}^{1}\le n\}),\mu_{1}^{2}(\{f_{1}^{2}\le n\})>0.
\]
Thus we may restrict $K$ further (and obtain new $A_{i}$, $\pi_{i}$,
and $\mu_{1}^{i}$) and assume that for $i=1,2$, $f_{1}^{i}(y)$
is bounded by $n$ for $\mu_{1}^{i}$-almost all $y\in M$. Since
each $A_{1}^{i}$, $i=1,2$, is bounded we obtain 
\[
\int\tilde{f}_{1}^{i}\log\tilde{f}_{1}^{i}d\m\in\mathbb{R}.
\]
Corollary \ref{cor:uniqueness-of-restrictions} implies that the $p$-optimal
dynamical coupling between $\mu_{0}$ and $\mu_{1}^{i}$ is still
unique so that the interpolation inequality implies 
\[
\log\m(\Gamma_{t}^{(i)})\ge(1-t)\log\m(K)-t\int f_{1}^{i}\log f_{1}^{i}d\m+KW_{p}(\mu_{0},\mu_{1}^{i})^{2}
\]
where $\Gamma^{(i)}=\supp\pi_{i}$ is the the support of the unique
$p$-optimal dynamical coupling of $\mu_{0}$ and $\mu_{0}^{i}$.
In particular, it holds 
\[
\lim_{t\to0}\m(\Gamma_{t}^{(i)})=\m(K)\quad\text{for }i=1,2.
\]

Also note that $\Gamma^{(1)}\cup\Gamma^{(2)}$ is $c_{p}$-cyclically
monotone, so that $(M,d,\m)$ being $p$-essential non-branching shows
for sequence $t_{n}\to0$ we may replace $\Gamma^{(1)}$ and $\Gamma^{(2)}$
by smaller sets $\tilde{\Gamma}^{(1)}\subset\Gamma^{(1)}$ and $\tilde{\Gamma}^{(2)}\subset\Gamma^{(2)}$
such that $\pi_{1}(\tilde{\Gamma}^{(1)})=\pi_{2}(\tilde{\Gamma}^{(2)})=1$
and $\m(\Gamma_{t_{n}}^{(1)}\cap\Gamma_{t_{n}}^{(2)})=0$ for all
large $n\in\mathbb{N}$. Note that still $\lim_{t\to0}\m(\tilde{\Gamma}_{t}^{(i)})=\m(K)$.
But then
\begin{align*}
\m(K) & \ge\lim_{\delta\to0}\m(K_{\delta})\\
 & \ge\lim_{t\to0}\m(\tilde{\Gamma}_{t}^{(1)}\cup\tilde{\Gamma}_{t}^{(2)})\\
 & =\lim_{t\to0}\left[\m(\tilde{\Gamma}_{t}^{(1)})+\m(\tilde{\Gamma}_{t}^{(2)})\right]\\
 & =2\m(K)
\end{align*}
which is a contradiction.
\end{proof}
\begin{thm}
Assume $(M,d,\m)$ is $p$-essentially non-branching and satisfies
the $\CD_{p}(K,\infty)$-condition. If $\mu_{0},\mu_{1}\in\mathcal{P}_{p}(M)$
such that there is a $p$-optimal dynamical coupling $\sigma\in\OptGeo_{p}(\mu_{0},\mu_{1})$
with $\mu_{0},(e_{t_{0}})_{*}\sigma\ll\m$ for some $t_{0}\in M$
then the $p$-optimal coupling $(e_{0},e_{1})_{*}\sigma$ is induced
by a transport map $T$.
\end{thm}
\begin{proof}
We reduce the general case to the lemma above. Let $t\mapsto\mu_{t}$
be a geodesic between $\mu_{0}$ and $\mu_{1}$. By Corollary \ref{cor:ess-nb-intermediate-transport}
we see that for fixed $t\in(0,1)$, $s\mapsto\mu_{st}$ is the unique
geodesic connecting $\mu_{0}$ and $\mu_{t}$ and there is a $p$-optimal
transport map $T_{t,1}$ between $\mu_{t}$ and $\mu_{1}$. Hence,
it suffices to show that there is a transport map from $\mu_{0}$
to $\mu_{t}$. Since $s\mapsto\mu_{st}$ is unique between its endpoint
it suffices to show the claim for $\mu_{0}$ and $\mu_{1}$ connected
by a unique geodesic $t\mapsto\mu_{t}$. 

Let $\sigma$ be the unique $p$-optimal dynamical coupling induced
by $t\mapsto\mu_{t}$ and $\pi=(e_{0},e_{1})_{*}\pi$. Denote the
densities of $\mu_{0}$ and $\mu_{1}$ with respect to $\m$ by $f_{0}$
and $f_{1}$ respectively. Since 
\[
\pi(\bigcup_{n\in\mathbb{N}}C_{n})=1
\]
where $C_{n}=(\{f_{0}\ge\frac{1}{n}\}\cap B_{n}(x_{0}))\times M$
for a fixed $x_{0}\in M$, it suffices to show that the claim holds
for $\pi$ restricted to $C_{n}$. Note by Corollary \ref{cor:uniqueness-of-restrictions}
the geodesic connecting the marginals of $\frac{1}{\pi(C_{n})}\pi\big|_{C_{n}}$
is still unique. 

Thus we can assume $\mu_{0}=f_{0}\m$ has bounded support with density
$f_{0}$ bounded below by an $\epsilon>0$ on a set $A$ of full $\mu_{0}$-measure
and $\sigma$ is the unique $p$-optimal dynamical coupling between
$\mu_{0}$ and $\mu_{1}$. Now for 
\[
f(x,y)=\chi_{A}(x)\frac{1}{f_{0}(x)},
\]
Corollary \ref{cor:uniqueness-of-restrictions} shows that $\sigma_{f}$
is still unique between $\mu_{0}^{f}=\frac{1}{\m(A)}\m\big|_{A}$
and $\mu_{1}^{f}$. It is easy to see that $(e_{0},e_{1})_{*}\sigma_{f}$
is induced by a transport map if and only if $(e_{0},e_{1})_{*}\sigma$
is induced by a transport map. We conclude by noticing that $A=\supp\mu_{0}^{f}$
and $\mu_{1}^{f}$ satisfy the conditions of the previous lemma and
hence $(e_{0},e_{1})_{*}\sigma_{f}$ is induced by a $p$-optimal
transport map.
\end{proof}
Since by Rajala\textendash Sturm \cite{RS2014NonbranchStrongCD} strong
$\CD_{p}(K,\infty)$-spaces are $p$-essentially non-branching, we
recover their result on the existence of transport maps \cite[Corollary 1.4]{RS2014NonbranchStrongCD}.
\begin{cor}
If $(M,d\m)$ satisfies the strong $\CD_{p}(K,\infty)$-condition
between every $\mu_{0},\mu_{1}\in\mathcal{P}_{p}(M)$ with $\mu_{0},\mu_{1}\ll\m$
there is a $p$-optimal dynamical coupling $\sigma$ and $(e_{0},e_{1})_{*}\sigma$
is induced by a transport map.
\end{cor}
If, instead, we know that between an absolutely continuous initial
measure and an arbitrary measure every interpolation is absolutely
continuousthen we can show general existence of transport maps.
\begin{cor}
Assume $(M,d\m)$ has the strong interpolation property $\sIP_{p}$.
Then $(M,d,\m)$ is $p$-essentially non-branching and satisfies the
(weak) $\CD_{p}(K,\infty)$-condition if and only if it satisfies
the strong $\CD_{p}(K,\infty)$-condition. Furthermore, if either
of the cases hold then $(M,d,\m)$ has good transport behavior $\GTB_{p}$
as well.
\end{cor}
In the more general setting there could be more than one geodesic
connecting two absolutely continuous measure. However, it is possible
to show that the absolutely continuous part of a geodesic connecting
measures with finite entropy is just a restricting of a unique absolutely
geodesic given by the interpolation inequality. Hence this geodesic
is unique.
\begin{cor}
Assume $(M,d,\m)$ is $p$-essentially non-branching and satisfies
the $\CD_{p}(K,\infty)$-condition. Then for all $\mu_{0},\mu_{1}\in\mathcal{P}_{p}(M)$
with $\mu_{0},\mu_{1}\ll\m$ and $\int f_{0}\log f_{0}d\m,\int f_{1}\log f_{1}d\m<\infty$
there is a unique $p$-optimal dynamical coupling $\sigma\in\OptGeo_{p}(\mu_{0},\mu_{1})$
along which the $\CD_{p}(K,\infty)$-interpolation inequality holds
and for this dynamical coupling the $p$-optimal coupling $(e_{0},e_{1})_{*}\sigma$
is induced by a transport map. 

Furthermore, for any other $p$-optimal dynamical coupling $\tilde{\sigma}\in\OptGeo_{p}(\mu_{0},\mu_{1})$
such that for some $t\in(0,1)$ the $t$-midpoint $\mu_{t}=\rho_{t}\m+\m$
with $\rho_{t}\not\equiv0$, there are a function $f:M\to[0,\infty)$
which is positive $\m$-almost everywhere on $\{\rho_{t}>0\}$ and
a Borel set $A_{t}\subset M$ such that 
\[
\sigma_{f}=\tilde{\sigma}_{\tilde{f}}
\]
where $\tilde{f}(\gamma)=\chi_{A_{t}}(\gamma_{t})$, i.e. the absolutely
continuous part of $\tilde{\sigma}$ is obtained via a restriction
of $\sigma$.
\end{cor}
\begin{proof}
Let $\sigma$ be a $p$-optimal dynamical coupling along which the
$\CD_{p}(K,\infty)$-inter\-polation inequality holds and $\tilde{\sigma}$
be another $p$-optimal dynamical coupling such that at time $t\in(0,1)$
the interpolation $\tilde{\mu}_{t}=(e_{t})_{*}\tilde{\sigma}$ has
density with respect to $\m$, i.e.
\[
\tilde{\mu}_{t}:=\tilde{\mu}_{t}^{a}+\tilde{\mu}_{t}^{s}
\]
with $\tilde{\mu}_{t}^{s}\bot\m$, $\tilde{\mu}_{t}^{a}\ll\m$ and
$\tilde{\mu}_{t}^{a}(M)>0$ . Let $A_{t}$ be a Borel set such that
$\tilde{\mu}_{t}^{a}(M\backslash A_{t})=0$ and $\tilde{\mu}_{t}^{s}(A_{t})=0$. 

Set $\tilde{f}(\gamma)=\chi_{A_{t}}(\gamma_{t})$, then by the theorem
above the $p$-optimal coupling $(e_{0},e_{t_{0}})_{*}\tilde{\sigma}_{\tilde{f}}$
between $\tilde{\mu}_{0}^{\tilde{f}}$ and $\tilde{\mu}_{0}^{\tilde{f}}$
which is induced by a transport map $\tilde{T}$. Similarly, $(e_{0},e_{t})_{*}\sigma$
is induced by a transport map $T$. 

We claim that $T=\tilde{T}$ on $A_{0}=\tilde{T}^{-1}(A_{t})$. Indeed,
this would imply the result because $\tilde{\mu}_{0}^{\tilde{f}}\le\mu_{0}\big|_{A_{0}}$.

The claim follows by observing that between $\hat{\mu}_{0}=\frac{1}{2}(\mu_{0}+\tilde{\mu}_{0}^{\tilde{f}})$
and $\hat{\mu}_{0}=\frac{1}{2}(\mu_{t}+\tilde{\mu}_{t}^{a})$ there
is a unique $p$-optimal coupling which is induced by a transport
map.
\end{proof}
\begin{rem*}
The corollary allows us to do localization so that we can show the
following: between any two measure $\mu_{0}=f_{0}\m$ and $\mu_{1}=f_{1}\m$
with finite entropy there is an interpolation $\mu_{t}=f_{t}d\m$
such that 
\[
\log f_{t}(\gamma_{t})\le(1-t)\log f_{0}(\gamma_{0})+t\log f_{1}(\gamma_{1})-Kd(\gamma_{0},\gamma_{1})^{2}.
\]
This can be used to obtain density bounds of $f_{t}$ and thus a weak
Poincaré inequality.
\end{rem*}

\section{\label{sec:Measure-rigidity}Measure rigidity}

In this section we prove the measure rigidity theorem, i.e. we will
show that two qualitatively non-degenerate measures on a $p$-essentially
non-branching space are mutually absolutely continuous.

For convenience of the reader we recall the main properties of $p$-essentially
non-branching, qualitatively non-degenerate measures $\m$.
\begin{thm*}
Let $(M,d,\m)$ be $p$-essentially non-branching and $\m$ be qualitatively
non-degenerate. Then for every $\mu_{0},\mu_{1}\in\mathcal{P}_{p}(M)$
with $\mu_{0}\ll\m$ there is a unique $p$-optimal dynamical coupling
$\sigma\in\OptGeo_{p}(\mu_{0},\mu_{1})$ and this dynamical coupling
satisfies $(e_{t})_{*}\sigma\ll\m$ for $t\in[0,1)$. 
\end{thm*}
Define 
\[
R_{t}(x):=e_{t}\left((e_{0},e_{t})^{-1}(\{x\}\times M)\right)
\]
to be the set of $t$-midpoints of geodesics starting at $x$. This
set is analytic and hence measurable. Furthermore, for $0<t\le s\le1$
it holds
\[
R_{t}(x)\subset R_{s}(x).
\]
Thus the following set
\[
R_{(0,t)}(x):=\bigcup_{0<s<t}R_{s}(x)=\bigcup_{s<t,s\in\mathbb{Q}}R_{s}(x)
\]
is also analytic and measurable. As an abbreviation we also write
$R_{(0,t]}(x)=R_{t}(x)$. Finally define the set of strict $t$-midpoints
by 
\[
R_{t}^{*}(x):=R_{t}(x)\backslash R_{(0,t)}(x).
\]

\begin{lem}
Assume $\m$ is non-degenerate. Then $\m(R_{t}^{*}(x))=0$.
\end{lem}
\begin{proof}
First note that 
\[
f(t)=\m(R_{(0,t]}(x)\cap B_{R}(x))
\]
is a non-decreasing function and finite for fixed $R>0$ so that for
some set $\Omega\subset[0,1]$ whose complement is at most countable
and it holds
\[
\m(R_{t}^{*}(x)\cap B_{R}(x))=\lim_{\epsilon\to0}f(t+\epsilon)-f(t)=0\quad\text{for all }t\in\Omega.
\]
Assume by contradiction $\Omega\ne(0,1]$. Then there is a $t\notin\Omega$
with $\m(R_{t}^{*}(x)\cap B_{R}(x))>0$. It always holds 
\[
(R_{t}^{*}(x)\cap B_{R}(x))_{s,x}\subset R_{st}^{*}(x)\cap B_{R}(x)\quad\text{for all }s\in(0,1).
\]
Also note that every $t\in(0,1]\backslash\Omega$ there is an $s\in(0,1)$
with $st\in\Omega$. In combination with the non-degenercity of $\m$
this leads to the following contradiction
\[
0=\m(R_{st}^{*}(x)\cap B_{R}(x))\ge m(R_{t}^{*}(x)\cap B_{R}(x))_{s,x}>0.
\]
\end{proof}
\begin{prop}
Assume $\m$ is a non-degenerate measure. For every measurable $A\subset M$
of finite measure and every $\epsilon>0$ there is a compact subset
$K\subset A$ with $\m(A\backslash K)<\epsilon$ and $t\in(0,1)$
such that 
\[
K\subset R_{t}(x).
\]
Furthermore, there is a $\m$-measurable map $T:K\to M$ such that
\[
d(x,T(y))=\frac{d(x,y)}{t}=d(x,y)+d(y,T(y)).
\]
In particular, there is a dynamical coupling $\sigma$ which is $p$-optimal
for all $p\in[1,\infty)$ such that $(e_{0})_{*}\sigma=\delta_{x}$,
$(e_{t})_{*}\sigma=\frac{1}{\m(K)}\m\big|_{K}$ and $(e_{1})_{*}\sigma=T_{*}((e_{t})_{*}\sigma)$.
\end{prop}
\begin{proof}
Just note that because $\m(R_{1}^{*}(x))=0$ and $t\mapsto R_{t}(x)$
is monotone we have
\[
\m(A)=\m(A\cap R_{1}(x))=\m(A\cap\cup_{t<1}R_{t}(x))=\lim_{t\to1}\m(A\cap R_{t}(x)).
\]
Now for all $\epsilon>0$ there is a $t\in(0,1)$ close to $1$ and
a compact set $K\subset A\cap R_{t}(x)$ such that $\m(A\backslash K)<\epsilon$.
Since $K\subset R_{t}(x)$ there is a measurable map $T:K\to M$ such
that 
\[
(e_{0},e_{t},e_{1})^{-1}(x,y,T(y))\ne\varnothing\qquad\text{for all }y\in K
\]
proving the first part of the claim. For the second just note that
$\mu_{t}=\frac{1}{\m(K)}\m\big|_{K}$ is a $t$-midpoint of $\delta_{x}$
and $T_{*}\mu_{t}$. %
\end{proof}
\begin{thm}
\label{thm:measure-rigidity-qual-non-deg}Assume $\m_{i}$, $i=1,2$,
are both $p$-essentially non-branching and qualitatively non-degenerate
measures on $(M,d)$. Then $\m_{1}$ and $\m_{2}$ are mutually absolutely
continuous. 
\end{thm}
\begin{proof}
Assume first there are two mutually singular measure $\m_{1}$ and
$\m_{2}$ satisfying the assumptions. Then we immediately arrive at
a contradiction: For measures $\mu_{0}\ll\m_{1}$ and $\mu_{1}\ll\m_{2}$
the interpolations must be absolutely continuous with respect to both
$\m_{1}$ and $\m_{2}$ which is not possible. 

For general measures $\m_{1}$ and $\m_{2}$ there is a maximal decomposition
\begin{align*}
\m_{1} & =\m+\m_{1}^{s}\\
\m_{2} & =f\m+\m_{2}^{s}
\end{align*}
for non-trivial mutually singular measures $\m$, $\m_{1}^{s}$ and
$\m_{2}^{s}$ and a non-negative $L_{\operatorname{loc}}^{1}(\m)$-function
$f$ which is positive $\m$-almost everywhere. 

Assume by contradiction $\m_{1}^{s}\not\equiv0$ and let $A$ be a
bounded set with $\m_{1}^{s}(A)>0$ and $\m_{2}(A)=0$. We claim that
$\m_{2}(A_{t,x})=0$ for all $x$ and $t\in(0,1)$. Indeed, if this
was not the case then for some $x\in M$ and $t\in(0,1)$ there is
a compact $K\subset A$ and $\sigma$ as in the previous proposition
such that $\m_{2}(K_{s,x})>0$ and $\m_{1}^{s}(K_{s,x})=0$ for $s\in(0,1)$.
In that case it holds $(e_{t})_{*}\sigma=\frac{1}{\m_{1}^{s}(K)}\m_{1}^{s}\big|_{K}$,
$(e_{st})_{*}\sigma\bot\m_{1}^{s}$ and $(e_{s,t})_{*}\sigma\ll\m_{2}$. 

However, by the strong interpolation property $\sIP_{p}$ between
the measures $(e_{st})_{*}\sigma$ and $(e_{1})_{*}\sigma$ this would
imply that $(e_{t})_{*}\sigma\ll\m_{2}$ which is a contradiction
as $\m_{2}$ and $\m_{1}^{s}$ are mutually singular.

This shows that 
\[
\m_{1}^{s}(A_{t,x})=\m_{1}(A_{t,x})\ge f(t)\m_{1}(A)=f(t)\m_{1}^{s}(A).
\]
Because $A$ is arbitrary, we see that $\m_{1}^{s}$ is qualitatively
non-degenerate and $p$-essentially non-branching. 

We arrive at a contradiction by observing that $\m_{1}^{s}$ and $\m_{2}$
are mutually singular.
\end{proof}
The following technical lemma can be extracted from the work of Cavalletti\textendash Huesmann
\cite{CH2014SelfInter}. For convenience of the reader we include
its short proof.
\begin{lem}
[Self-Intersection Lemma] Assume $t\mapsto\mu_{t}=f_{t}\m$ is a
geodesic in $\mathcal{P}_{p}(M)$ such that $\mu_{0}=\frac{1}{\m(K)}\m\big|_{K}$
for some compact set $K$, $\supp\mu_{1}$ having bounded support
and it holds $C:=\sup_{t\in[0,\delta]}\|f_{t}\|_{\infty}<\infty$
for $\delta\in(0,1)$ then there is a $t_{0}\in(0,\delta)$ such that
for all $t\in[0,t_{0})$ it holds $\mu_{t}(K)>0$. In particular,
$\mu_{t}$ and $\mu_{0}$ cannot be mutually singular.
\end{lem}
\begin{proof}
Assume this is not the case then there is a sequence $t_{n}\to0$
such that $\mu_{t_{n}}\bot\mu_{0}$. In particular, there are Borel
sets $A_{0}\subset K$ and $A_{n}\subset\supp\mu_{t_{n}}$ with $A_{n}\cap A_{0}=\varnothing$
and $\mu_{t_{n}}(A_{n})=\mu_{0}(A_{0})=1$. Note that this shows $\m(A_{0})=\m(K)$.

Since the support of $\mu_{0}$ and $\mu_{t}$ are bounded for all
$\epsilon>0$ there is a $t_{\epsilon}\in(0,\delta)$ such that all
$t_{n}\le t_{\epsilon}$ 
\[
A_{n}\subset\supp\mu_{t_{n}}\subset K_{\epsilon}.
\]
Also note that 
\[
\m(A_{n})\ge\frac{1}{C}\mu_{t_{n}}(A_{n})=\frac{1}{C}
\]
But then 
\begin{align*}
\m(K) & =\lim_{\epsilon\to0}\m(K_{\epsilon})\\
 & \ge\limsup_{n\to\infty}\m(A_{0}\cup A_{n})\\
 & \ge\m(K)+\limsup_{n\to\infty}\m(A_{n})\ge\m(K)+\frac{1}{C}
\end{align*}
which is a contradiction.
\end{proof}
\begin{thm}
\label{thm:measure-rigidity-ess-nb-CD}Assume $(M,d,\m_{1})$ and
$(M,d,\m_{2})$ have the strong interpolation property $\sIP_{p}$
and satisfy the $\CD_{p}(K,\infty)$-condition. Then $\m_{1}$ and
$\m_{2}$ are mutually absolutely continuous. 
\end{thm}
\begin{rem*}
The proof below also works in the setting of $p$-essentially non-branching,
qualitatively non-degenerate measures. Indeed, by Proposition \ref{prop:qND-ess-nb-implies-bdd-dens}
the density $f_{t}$ of the $t$-interpolation $\mu_{t}$ is uniformly
bounded by the density of $f_{0}$ if $t$ is close to $1$, compare
this also to \cite[Theorem 4.2]{Rajala2012a} and \cite[Theorem 4.1]{CM2016TransMapsMCP}. 

However, the strong interpolation property $\sIP_{p}$, which follows
from qualitative non-degenericity, is essential in order to combine
the singular part of one of the measures with the bounded density
property. It is unclear whether without this property there could
be more than one $\CD_{p}(K,\infty)$-measure.
\end{rem*}
\begin{proof}
As above two such measures $\m_{1}$ and $\m_{2}$ cannot be mutually
singular and must be of the form 
\begin{align*}
\m_{1} & =\m+\m_{1}^{s}\\
\m_{2} & =f\m+\m_{2}^{s}
\end{align*}
for non-trivial mutually singular measures $\m$, $\m_{1}^{s}$ and
$\m_{2}^{s}$ and a non-negative $L_{\operatorname{loc}}^{1}(\m)$-function
$f$ which is positive $\m$-almost everywhere.

Assume by contradiction $\m_{1}^{s}\not\equiv0$ and let $A$ and
$B$ be compact sets with $\m_{1}^{s}(A)>0$, $\m_{2}(A)=0$, $\m_{1}^{s}(B)=0$,
$\m_{2}^{s}(B)=0$ and $\m(B)>0$. Let $\mu_{0}=\frac{1}{\m_{1}^{s}(A)}\m_{1}^{s}\big|_{A}$
and $\mu_{1}=\frac{1}{\m(B)}\m\big|_{B}$. Then there is a unique
geodesic $t\mapsto\mu_{t}$ connecting $\mu_{0}$ and $\mu_{1}$ which
is absolutely continuous with respect to $\m_{1}$. 

Also note by the strong interpolation property for $\m_{2}$, $\mu_{1}\ll\m_{2}$
implies $\mu_{t}\ll\m_{2}$ for $t\in(0,1)$. In particular, $\mu_{t}(A)=0$
for all $t\in(0,1]$. Now the $\CD_{p}(K,\infty)$-condition implies
\begin{align*}
\|f_{t}\|_{\infty} & \le C(K,\operatorname{diam}A,\operatorname{diam}B)\cdot\max\{\|f_{0}\|_{\infty},\|f_{1}\|_{\infty}\}.\\
 & =C(K,\operatorname{diam}A,\operatorname{diam}B)\cdot\max\{\frac{1}{\m_{1}^{s}(A)},\frac{1}{\m(B)}\}
\end{align*}
where $\mu_{t}=f_{t}\m$.

We arrive at a contradiction by observing that $\mu_{t}(A)>0$ by
the Self-Intersection-Lemma above.
\end{proof}

\appendix

\section{Proof of Theorem \ref{thm:ess-nb-summary} and Corollary \ref{cor:ess-nb-intermediate-transport}}

Before we prove the theorem we need the following technical lemmas. 

\begin{lem}
\label{lem:nb-implies-transport-to-delta}Let $\sigma$ be a $p$-optimal
dynamical coupling $\mu_{0}$ and $\mu_{1}$ such that $(e_{t})_{*}\sigma=\delta_{x_{t}}$
for some $t\in(0,1)$ and $x_{t}\in M$ then $\mu_{0}\otimes\mu_{1}$
is a $p$-optimal coupling and $d(\cdot,\cdot)$ is constant on $\supp\mu_{0}\times\supp\mu_{1}$.
In particular, if $\mu_{0}\otimes\mu_{1}$ is not a delta measure
then there is a $p$-optimal dynamical coupling $\tilde{\sigma}$
such that $\tilde{\sigma}(\mathsf{L})<1$ for all measurable non-branching
$\mathsf{L}\subset\Geo_{[0,1]}(M,d)$ of non-branching geodesics.
\end{lem}
\begin{proof}
First note that the trivial coupling $\mu_{0}\otimes\mu_{1}$ is a
$p$-optimal coupling of $\mu_{0}$ and $\mu_{1}$ for all $p\in[1,\infty)$.
Indeed, the assumptions imply that for each $\gamma,\eta\in\supp\sigma$
it holds $\gamma_{t}=\eta_{t}=x_{t}$ and hence $\ell(\gamma)=\ell(\eta)$.
But then $\ell(\restr_{0,t}\gamma)=W_{p}(\mu_{0},\delta_{x_{t}})$
and $\ell(\restr_{t,1}\gamma)=W_{p}(\delta_{x_{t}},\mu_{1})$ and
thus for all $x_{0}\in\supp\mu_{0}$ and $x_{1}\in\supp\mu_{1}$
\begin{align*}
d(x_{0},x_{1}) & =\frac{d(x_{0},x_{t})}{t}=\frac{W_{p}(\mu_{0},\delta_{x_{t}})}{t}=W_{p}(\mu_{0},\mu_{1}).\\
 & =\frac{d(x_{t},x_{1})}{1-t}=\frac{W_{p}(\delta_{x_{t}},x_{1})}{1-t}=W_{p}(\mu_{0},\mu_{1})
\end{align*}
implying that $\supp\mu_{0}\times\supp\mu_{1}$ is $c_{p}$-cyclically
monotone.

Let $T_{0,t}:M\to\Geo_{[0,1]}(M,d)$ be a measurable map such that
$T_{0,t}(x_{0})$ is a geodesic between $x_{0}$ and $x_{t}$. Then
$\sigma_{0,t}:=(T_{0,t})_{*}\mu_{0}$ is a $p$-optimal dynamical
coupling between $\mu_{0}$ and $\delta_{x_{t}}$. Similarly, let
$T_{t,1}:M\to\Geo_{[0,1]}(M,d)$ be a Borel map such that $T_{t,1}(x_{1})$
is a geodesic between $x_{t}$ and $x_{1}$. Note that for each $x_{0}\in\supp\mu_{0}$
and $x_{1}\in\supp\mu_{1}$ 
\[
\gamma_{s}^{x_{0},x_{1}}=\begin{cases}
T_{0,t}(x_{0})\left(\frac{s}{t}\right) & s\in[0,t]\\
T_{t,1}(x_{1})\left(\frac{s-t}{1-t}\right) & s\in[t,1]
\end{cases}
\]
is a geodesic between $x_{0}$ and $x_{1}$. Since $T_{0,t}$ and
$T_{t,1}$ are Borel maps, so is $T_{0,1}^{x_{t}}:(x_{0},x_{1})\mapsto\gamma^{x_{0},x_{1}}$.
In particular, $\tilde{\sigma}=(T_{0,1}^{x_{t}})_{*}\mu_{0}\otimes\mu_{1}$
is a $p$-optimal dynamical coupling of $\mu_{0}$ and $\mu_{1}$. 

If $\mu_{0}\otimes\mu_{1}$ is not a delta measure then either $\mu_{0}$
or $\mu_{1}$ (or both) is not a delta measure. Assume $\mu_{1}$
is not a delta measure then for each set $\mathsf{L}$ with $\tilde{\sigma}(\Gamma)=1$
and for $\mu_{0}$-almost all $x_{0}\in e_{0}(\Gamma)$ there are
at least two distinct geodesics $\gamma,\eta\in\Gamma$ with 
\[
\restr_{0,t}\gamma=\restr_{0,t}\eta=T_{0,t}(x_{0}).
\]
In particular, $\mu_{0}\otimes\mu_{1}$ is not concentrated on a non-branching
set.
\end{proof}
\begin{rem*}
Assume $x\mapsto(\mu_{0}^{x},\mu_{1}^{x})$ is a measurable map such
that $\delta_{x}$ is the $t$-midpoint of $\mu_{0}^{x}$ and $\mu_{1}^{x}$.
Then $x\mapsto(T_{0,1}^{x})_{*}(\mu_{0}^{x}\otimes\mu_{1}^{x})$ is
also measurable. 
\end{rem*}
\begin{lem}
\label{lem:nb-end-implies-nb-mid}Let $\mu_{0}$ and $\mu_{1}$ be
probability measures such that any $p$-optimal dynamical coupling
between $\mu_{0}$ and $\mu_{1}$ is concentrated on a set of non-branching
geodesics. Then for any $t$-midpoint $\mu_{t}$ of $\mu_{0}$ and
$\mu_{1}$, any $p$-optimal dynamical coupling between $\mu$ and
$\mu_{t}$ is concentrated on a set of non-branching geodesics.
\end{lem}
\begin{proof}
It is easy to see that any $p$-optimal dynamical coupling $\sigma\in\OptGeo_{p}(\mu,\mu_{t})$
is obtained by restricting a $p$-optimal dynamical coupling $\tilde{\sigma}\in\OptGeo_{p}(\mu,\nu)$,
i.e. 
\[
(\restr_{0,t})_{*}\tilde{\sigma}=\sigma
\]
and hence
\[
(\restr_{0,t})_{*}\OptGeo_{p}(\mu_{0},\mu_{1})=\OptGeo_{p}(\mu_{0},\mu_{t}).
\]
Furthermore, if $\mathsf{L}$ is non-branching and measurable then
$\restr_{0,t}(\mathsf{L})$ is also non-branching and measurable.
In particular, choosing $\mathsf{L}$ such that $\tilde{\sigma}(L)=1$
implies $\sigma(\restr_{0,t}(\mathsf{L}))=1$.
\end{proof}
\begin{prop}
\label{prop:nb-implies-transport-to-delta}Let $\mu_{0}$ and $\mu_{1}$
be probability measures such that any $p$-optimal dynamical coupling
between $\mu_{0}$ and $\mu_{1}$ is concentrated on a set of non-branching
geodesics. Then for any $p$-optimal dynamical coupling $\sigma\in\OptGeo_{p}(\mu_{0},\mu_{1})$
and any $t\in(0,1)$ and $s\in[0,1]$ the $p$-optimal coupling $(e_{t},e_{s})_{*}\sigma$
are induced by a transport map. 
\end{prop}
\begin{proof}
Let $\mu_{s}=(e_{s})_{*}\sigma$, $\pi_{t,0}=(e_{t},e_{0})_{*}\sigma$
and $\pi_{t,1}=(e_{t},e_{1})_{*}\sigma$. It suffices to show that
$\pi_{t,0}$ is induced by a transport map $T$, i.e. $(\operatorname{id}\times T)_{*}\mu_{t}=\pi_{t,0}$.
By disintegrating $\pi_{t,0}$ and $\pi_{t,1}$ over $(e_{t},e_{0})$
and resp. $(e_{t},e_{1})$ we get 
\begin{align*}
\pi_{0,t} & =\int\mu_{x}\otimes\delta_{x}d\mu_{t}(x)\\
\pi_{t,1} & =\int\delta_{x}\otimes\nu_{x}d\mu_{t}(x).
\end{align*}
 Let $\sigma=\int\sigma_{x}d\mu_{t}(x)$ denote the disintegration
of $\sigma$ over $e_{t}$ and define a new dynamical coupling 
\[
\tilde{\sigma}=\int(T_{0,1}^{x})_{*}(\mu_{x}\otimes\nu_{x})d\mu_{t}(x)
\]
where $T_{0,1}^{x}$ is defined as in the proof of the previous lemma.
Note that $(x_{0},x_{1},x_{t})\mapsto T_{0}^{x_{t}}(x_{0},x_{1})$
is measurable on $(e_{0},e_{1},e_{t})(\Geo_{[0,1]}(M,d))$ and hence
$x\mapsto(T_{0,1}^{x})_{*}(\mu_{x}\otimes\nu_{x})$ measurable on
$\supp\mu_{t}$. 

We claim $\tilde{\sigma}$ is $p$-optimal. Indeed, by the previous
lemma it holds 
\begin{align*}
\int d(\gamma_{0},\gamma_{1})^{p}d\tilde{\sigma}(\gamma) & =\int\int\int d(y,z)^{p}d\mu_{x}(y)d\nu_{x}(z)d\mu_{t}(x)\\
 & =\int\frac{1}{t^{p}}d(y,x)^{p}d\mu_{x}(y)d\mu_{t}(x)\\
 & =\frac{1}{t^{p}}\int d(y,x)^{p}d\pi_{0,t}(y,x)=W_{p}(\mu,\nu)^{p}.
\end{align*}

If $\pi_{0,t}$ is not induced by a transport map then there is a
Borel set $B\subset M$ of positive $\mu_{t}$-measure such that for
all $x\in B$ the measure $\mu_{x}\otimes\nu_{x}$ is not a delta
measure. This, however, implies that for all $x\in B$ the measure
$\tilde{\sigma}_{x}=(T_{0,1}^{x})_{*}(\mu_{x}\otimes\nu_{x})$ is
not concentrated on a set of non-branching geodesics. The assumption
shows that there is a non-branching measurable set $\mathsf{L}\subset\Geo_{[0,1]}(M,d)$
with $\tilde{\sigma}(\mathsf{L})=1$. But this is a contradiction
since $\sigma_{x}(\mathsf{L})<1$ for $\mu_{t}$-almost all $x\in B$
implies 
\[
\tilde{\sigma}(\mathsf{L})=\int\tilde{\sigma}_{x}(\mathsf{L})d\mu_{t}(x)<1.
\]
\end{proof}
\begin{cor}
Let $\mu_{0}$ and $\mu_{1}$ be probability measures such that any
$p$-optimal dynamical coupling between $\mu_{0}$ and $\mu_{1}$
is concentrated on a set of non-branching geodesics. Then for any
$p$-optimal dynamical coupling $\sigma\in\OptGeo_{p}(\mu_{0},\mu_{1})$
and any $t\in(0,1)$ there is a Borel map $\mathsf{T}_{t}:M\to\Geo_{[0,1]}(M,d)$
such that 
\[
\sigma=\int\delta_{\mathsf{T}_{t}(x)}d\mu_{t}(x)
\]
where $\mu_{t}=(e_{t})_{*}\sigma$. In particular, whenever $\tilde{\sigma}\in\OptGeo_{p}(\mu_{0},\mu_{1})$
with $\mu_{t}=(e_{t})_{*}\tilde{\sigma}$ then $\sigma\equiv\tilde{\sigma}$.
\end{cor}
\begin{proof}
Let $\sigma=\int\sigma_{x}d\mu_{t}(x)$ be the disintegration of $\sigma$
over $e_{t}$. The proof above shows that $\sigma$ is unique among
all dynamical couplings $\tilde{\sigma}\in\OptGeo_{p}(\mu_{0},\mu_{1})$
with $\mu_{t}=(e_{t})_{*}\tilde{\sigma}$. For fixed $s\in[0,1]$
there is a transport map $T_{t,s}$ such that 
\[
(e_{t},e_{s})_{*}\sigma=\int(e_{t},e_{s})_{*}\sigma_{x}d\mu_{t}(x)=\int\delta_{x}\otimes\delta_{T_{t,s}(x)}d\mu_{t}(x).
\]
In particular, there is a Borel set $\Omega_{s}\subset M$ with $\mu_{t}(\Omega_{s})=1$
and 
\[
(e_{t},e_{s})_{*}\sigma_{x}=\delta_{x}\otimes\delta_{T_{t,s}(x)}
\]
for all $x\in\Omega_{s}$. Let $(s_{n})_{n\in\mathbb{N}}$ be dense
in $(0,1)$ and note $\mu_{t}(\Omega)=1$ where $\Omega=\cap_{n\in\mathbb{N}}\Omega_{s_{n}}$.
Define $\gamma_{s}^{x}=T_{t,s}(x)$and observe that $\gamma^{x}\in\Geo_{[0,1]}(M,d)$
and 
\[
(e_{t},e_{s_{n}})_{*}\sigma_{x}=(e_{t},e_{s_{n}})\delta_{\gamma^{x}}.
\]
This shows that $\sigma_{x}=\delta_{\gamma^{x}}$ on $\Omega$ and
thus 
\[
\sigma=\int\delta_{\mathsf{T}_{t}(x)}d\mu_{t}(x)
\]
where $\mathsf{T}_{t}:M\to\Geo_{[0,1]}(M,d)$ is any measurable map
with $\mathsf{T}_{t}(x)=\gamma^{x}$ on $\Omega$.
\end{proof}
\begin{proof}
[Proof of Theorem \ref{thm:ess-nb-summary}] By the $p$-essentially
non-branching property we see that the second statement follows directly
from the previous corollary. Furthermore, for $\Omega$ as in the
previous proof we can choose 
\[
\mathsf{L}=\supp\sigma\cap e_{t}^{-1}(\Omega)\subset\Geo_{[0,1]}(M,d).
\]
Then $\sigma$ is concentrated on $\mathsf{L}$ and whenever $\gamma,\eta\in\mathsf{L}$
with $\gamma_{t}=\eta_{t}$ then $\gamma_{t},\eta_{t}\in\Omega$ so
that $\mathsf{T}_{t}(\gamma_{t})\equiv\gamma\equiv\eta$. 
\end{proof}
\begin{proof}
[Proof of Corollary \ref{cor:ess-nb-intermediate-transport}] Assume
$t\mapsto\tilde{\mu}_{s}^{0,t}$ and $t\mapsto\tilde{\mu}_{s}^{t,1}$
is a geodesic connecting $\mu_{0}$ and $\mu_{t}$ and resp. $\mu_{t}$
and $\mu_{1}$. Then 
\[
t\mapsto\hat{\mu}_{s}=\begin{cases}
\tilde{\mu}_{\frac{s}{t}}^{0,t} & s\in[0,t]\\
\tilde{\mu}_{\frac{s-t}{1-t}}^{t,1} & s\in[t,1]
\end{cases}
\]
is also a geodesic connecting $\mu_{0}$ and $\mu_{1}$. Denote the
induced $p$-optimal dynamical coupling by $\hat{\sigma}$ and note
that $\mathring{\sigma}=\frac{1}{2}\left(\sigma+\hat{\sigma}\right)$
is also a $p$-optimal dynamical coupling between $(e_{0})_{*}\sigma$
and $(e_{1})_{*}\sigma$. Thus it holds
\[
\mathring{\sigma}=\int\delta_{\mathring{\mathsf{T}}_{t}(x)}d\mu_{t}(x)=\frac{1}{2}\int\delta_{\mathsf{T}_{t}(x)}+\delta_{\tilde{\mathsf{T}}_{t}(x)}d\mu_{t}(x)
\]
where $\mathsf{T}_{t}$, $\tilde{\mathsf{T}}_{t}$ and $\mathring{\mathsf{T}}_{t}$
are the maps given in the theorem. This, however, shows that the three
maps agree $\mu_{t}$-almost everywhere implying $\hat{\sigma}=\sigma$
and thus $\mu_{s}=\hat{\mu}_{s}$for $s\in[0,1]$. In particular,
$t\mapsto\mu_{ts}$ and $t\mapsto\mu_{s+(1-t)}$ are the unique geodesic
between $\mu_{0}$ and $\mu_{t}$ and resp. $\mu_{t}$ and $\mu_{1}$.
\end{proof}

\bibliographystyle{amsalpha}
\bibliography{bib}

\end{document}